\documentclass[11 pt]{amsart}
\usepackage{amssymb,color}
\usepackage{latexsym,pdfsync}
\usepackage{epstopdf}
\usepackage{graphicx}
\usepackage{amsmath}
\usepackage{amsfonts}
\usepackage[numbers]{natbib}
\usepackage{enumerate}
\usepackage{url}

\newcommand{\bx}{\boldsymbol x}
\allowdisplaybreaks
\iftrue 
\usepackage{amsmath}
\usepackage{amsfonts}
\newcommand{\field}[1]{\mathbb{#1}}
\DeclareMathOperator{\PR}{\field{P}}             
\DeclareMathOperator{\E}{\field{E}}              
\def\N{\field{N}}                                
\def\R{\field{R}}                                
\def\F{\field{F}}                                

\else
\def\PR{\mathop{\rm I\kern -0.20em P}\nolimits}  
\def\E{\mathop{\rm I\kern -0.20em E}\nolimits}   
\def\N{\mathop{\rm I\kern -0.20em N}\nolimits}   
\def\R{\mathop{\rm I\kern -0.20em R}\nolimits}   
\def\F{\mathop{\rm I\kern -0.20em F}\nolimits}   
\fi

\newtheorem{thm}{Theorem}[section]
\newtheorem{cor}[thm]{Corollary}
\newtheorem{lem}[thm]{Lemma}
\newtheorem{prop}[thm]{Proposition}
\newtheorem{ex}[thm]{Example}
\newtheorem{rem}[thm]{Remark}

\numberwithin{equation}{section}

\title[ Hidden regular variation]{Hidden regular variation: Detection and Estimation} 

\author[ A.\ Mitra ]{Abhimanyu Mitra}
\address{Abhimanyu Mitra\\ 206 Rhodes Hall\\ School of OR\&IE, Cornell University,
Ithaca, NY-14853\\ Tel: 1-607-255-2981} \email{am492@cornell.edu}

\author[ S. I.\ Resnick ]{Sidney I. Resnick}
\address{Sidney I. Resnick\\ 206 Rhodes Hall\\School of OR\&IE, Cornell University,
Ithaca, NY-14853 \\ Tel: 1-607-255-1210} \email{sir1@cornell.edu}

\normalsize

\begin{document}

\begin{center}
\maketitle

\end{center}
\normalsize

\begin{abstract}
{\it{Hidden regular variation}} defines a subfamily of distributions
satisfying  regular variation on $\mathbb{E} = [0,
\infty]^d \backslash \{(0,0, \cdots, 0) \} $ and models another
 regular variation on the sub-cone $\mathbb{E}^{(2)}
= \mathbb{E} \backslash \cup_{i=1}^d \mathbb{L}_i$,
where $\mathbb{L}_i$ is the $i$-th axis. We extend the concept of
{\it{hidden regular variation}} to sub-cones of $\mathbb{E}^{(2)}$ as
well. We suggest a procedure for detecting the presence of {\it{hidden
    regular variation}}, and if it exists, propose a method of
estimating the limit measure exploiting its semi-parametric
structure. We
 exhibit examples where {\it{hidden regular variation}}
 yields better estimates of probabilities of risk sets.   
\vskip 0.2 cm

{\it{Keywords}}: Regular variation; vague convergence; weak convergence;
  spectral measure; risk sets.

\end{abstract}

\section{Introduction} \label{intro}
Multivariate risks with Pareto-like tails 
 are usually modeled using the theory of
regular variation on cones. Let $ \mathbb{C}$ be a cone in $[-\infty,
\infty]^d$ satisfying $\bx \in \mathbb{C}$ implies $t\bx \in
\mathbb{C}$ for $t >0$ and denote the set of all non-negative Radon measures on
$\mathbb{C}$ by $M_+(\mathbb{C}).$ The distribution of a random vector
${\bf{Z}}$ is regularly varying on $\mathbb{C}$ if there exist a
scaling function $g(t) \uparrow \infty,$ and a non-zero Radon measure
$\chi \in M_+(\mathbb{C})$ such that 
\begin{equation}\label{reg_var_on_general_cone}
tP\left[ \frac{{\bf{Z}}}{g(t)} \in \cdot \right] \stackrel{v}{\rightarrow} \chi(\cdot)
\end{equation}
in $M_+(\mathbb{C}),$ where $\stackrel{v}{\rightarrow}$ denotes vague
convergence \cite{resnick:2008}. Risks with heavy tails could also be
modeled by stable distributions on a general convex cone; see
\cite{davydov:molchanov:zuyev:2007}.

Suppose the distribution of a random vector ${\bf{Z}}$ is regularly varying on the first quadrant
  $\E: = [0,  \infty]^d\backslash \{(0,0, \cdots 0) \}$ as in
\eqref{reg_var_on_general_cone} with limit measure $\nu$. It is
possible for $\nu$ to give  zero mass to a proper sub-cone
$\mathbb{C} {\subset} \E;$  for example, we could have
$$\mathbb{C}=
\E^{(2)} = \mathbb{E} \backslash \cup_{1 \le
  j_1< j_2 < \cdots < j_{d-1} \le d} \{x^{j_1} = 0, \cdots,
x^{j_{d-1}} = 0 \},
$$ 
the first quadrant with the axes removed.
If the distribution of $\bf{Z}$ is also
regularly varying on the subcone $\mathbb{C}$
 with scaling function $g_{\mathbb{C} } (t)\uparrow
\infty$ and $g(t)/g_{\mathbb{C}}(t) \to \infty$, then we say the
distribution of $\bf{Z}$ possesses
{\it hidden regular
variation (HRV)\/} on $\mathbb{C}$.
HRV
helps  detect  finer structure that may be ignored by
 regular variation on $\E.$ We will later refine our definition of hidden regular variation for a finite sequence of cones $\E \supset \mathbb{C}_1 \supset \mathbb{C}_2 \supset \cdots \mathbb{C}_m$.

Failure of  regular variation on $\E$ to distinguish
between independence and asymptotic independence prodded Ledford and Tawn
\cite{ledford:tawn:1996, ledford:tawn:1997} to define the
{\it{coefficient of tail dependence}} and this idea was extended to
hidden regular variation on $\E^{(2)}$ in \cite{resnick:2002a}. 
See also \cite{coles:heffernan:tawn:1999, dehaan:deronde:1998,
  draisma:drees:ferreira:dehaan:2004, heffernan:resnick:2005,
  campos:marron:resnick:jaffay:2005, maulik:resnick:2005, peng:1999,
  poon:rockinger:tawn:2003, ramos:ledford:2009, resnick:2004,
  starica:1999}.

Hidden regular variation provides  models that 
possess  regular variation on $\E$ and asymptotic
independence \citep[pages 323-325]{resnickbook:2007}. 
The concept has
typically been  considered in two dimensions
using the sub-cone $\E^{(2)}$. It is not clear how best to  extend the ideas of HRV to 
 dimensions  higher than two and one obvious remark  is that 
how one proceeds with definitions depends on the  sort of risk regions
being considered.

To demonstrate what is possible in higher dimensions, in this paper
we define
hidden regular variation on the sub-cones 
$$
\E^{(l)} = [0, \infty]^d \backslash \cup_{1 \le j_1< j_2 < \cdots <
  j_{d-l+1} \le d} \{x^{j_1} = 0, \cdots, x^{j_{d-l+1}} = 0 \}
,\quad 3 \le l \le d, $$ of
$\E$
and show with an example that 
asymptotic independence is not a necessary condition for HRV
 on $\E^{(l)}, \hskip 0.1 cm 3 \le l \le d.$
Hidden regular variation on $\E^{(l)},$ means that the
distribution of the random vector ${\bf{Z}}$ is 
 regularly varying on $\E$ as in
\eqref{reg_var_on_general_cone} with limit measure $\nu$ and
$\nu(\E^{(l-1)}) > 0,$ but  $\nu(\E^{(l)}) = 0.$ Also,
there is a scaling function $g_{\E^{(l)}} (t)$ satisfying
  $g(t)/g_{\E^{(l)}} (t) \to \infty$ which makes the distribution of 
${\bf{Z}}$  regularly varying on the cone $\E^{(l)}$ as in
\eqref{reg_var_on_general_cone} with limit measure $\nu^{(l)}$. Later,
when we define HRV on the finite sequence of cones $\E \supset
\E^{(2)} \supset \cdots \supset \E^{(d)},$ our definition of HRV on
$\E^{(l)}$ will be modified accordingly. 
We suggest 
exploratory methods for detecting the presence of hidden regular
variation on $\E^{(l)}, \hskip 0.1 cm 2 \le l \le d.$ The existing
method of detecting hidden regular variation on $\E^{(2)}$ is valid
only for dimension $d = 2,$ but our detection methods are applicable
for any finite dimension. 

If exploratory detection methods 
confirm
 data is consistent with the hypothesis of  
regular variation on a cone $\E^{(l)}$ as in
\eqref{reg_var_on_general_cone},  we must estimate 
 the limit measure $\nu^{(l)}$. Previous methods
 \citep{heffernan:resnick:2005} for estimating the
limit measure $\nu^{(2)}$ of hidden regular variation on
$\mathbb{E}^{(2)}$ 
have been non-parametric and ignored the semi-parametric structure
of $\nu^{(2)}.$
We offer some improvement
by exploiting the semi-parametric structure of $\nu^{(2)}$ 
and  estimate the parametric and
non-parametric parts of $\nu^{(2)}$ separately. 

On $\mathbb{E}$, estimation of the limit measure of  regular variation
is resolved by the familiar method of
the polar coordinate transformation ${\bf{x}} \to (|| {\bf{x}}||,
{{\bf{x}}}/{|| {\bf{x}}||});$  after this transformation, the  limit measure
$\nu$  is a product of a
probability measure $S$ and a Pareto measure $\nu_{\alpha}$,
$\nu_{\alpha}((r, \infty]) = r^{-\alpha}, r > 0$ \citep[pages
168-179]{resnickbook:2007}. Trying  to decompose
$\nu^{(2)}$ in this way presents the difficulty that the decomposition gives a
Pareto measure $\nu_{\alpha^{(2)}}$ and a possibly infinite Radon
measure \cite[pages 324-339]{resnickbook:2007}. 
So we 
transform  to a different coordinate system after which
$\nu^{(2)}$ is a product of a Pareto measure
$\nu_{\alpha^{(2)}}$ and a probability measure $S^{(2)}$ on $\delta
\aleph^{(2)} = \{ {\bf{x}} \in \E^{(2)}: x^{(2)} = 1\}$, where
$x^{(2)}$ is the second largest component of $\bx$.  We call the
probability measure $S^{(2)}$ the {\it hidden angular measure\/}
on $\E^{(2)}.$ We 
suggest procedures for consistently estimating the parameter $\alpha^{(2)}$ of the
Pareto measure $\nu_{\alpha^{(2)}}$ and the hidden angular measure $S^{(2)}$ 
and explain how these estimates lead to an estimate of
$\nu^{(2)}$.
If HRV on 
$\E^{(l)}$ is present for some  $3 \le l \le d,$   there is a similar
transformation of  coordinates making $\nu^{(l)}$ a
product of a Pareto measure $\nu_{\alpha^{(l)}}$ and a probability
measure $S^{(l)}$ on $\delta \aleph^{(l)} = \{ {\bf{x}} \in \E^{(l)}:
x^{(l)} = 1\},$ where $x^{(l)}$ is the $l$-th largest component of
  $\bx$.  We call  this probability measure $S^{(l)}$, the hidden
angular measure on $\E^{(l)}$ and employ similar estimation
methods for $l \geq 3$ as we did for $l=2$.

 For empirical exploration of the spectral or hidden spectral
 measures, it is often
desirable to make density plots. However, the hidden spectral measure
$S^{(l)}$ is supported on $\delta \aleph^{((l)}$, which is a difficult
plotting domain. For example, when $d =3,$ the set $\delta
\aleph^{(2)}$ is a disjoint union of six rectangles lying on three
different planes as shown in Figure \ref{deltan2}. Though $\delta 
\aleph^{(l)}$ is a $(d-1)$-dimensional set, $d$-dimensional vectors are needed to represent $\delta
\aleph^{(l)}$. So, the density plots on $\delta
\aleph^{(l)}$ also requires an additional dimension. In the two dimensional case, the problem is resolved by taking a transformation of points from $\delta
\aleph^{(1)} =  \{ {\bf{x}} \in \E : x^{(1)} = 1\}$ to $[0,1]$ and looking at the density of the induced probability measure of the transformed points \cite[ pages 316-321]{resnickbook:2007}. 
 We seek similar
appropriate transformations in higher dimensional cases. 
We devise a transformation of points from $\delta
\aleph^{(l)}$
to the $(d-1)$-dimensional simplex $\Delta_{d-1} = \{ {\bf{x}} \in
[0,1]^{d-1}: \sum_{i=1}^{d-1} x^i \le 1\}$ (see Section \ref{diff_par}). The probability measure $\tilde S^{(l)}$ on the transformed points induced by $S^{(l)}$ is called the transformed (hidden) spectral measure. Since the set $\Delta_{d-1}$ is represented by $(d-1)$-dimensional vectors, the problem of incorporating an additional dimension in the density plots vanishes.

\begin{figure}
\centering
\scalebox{0.6}
{\includegraphics{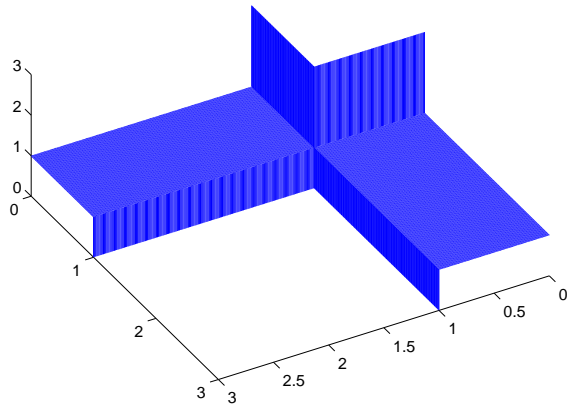}}
\caption{ The shaded region shows the set $\delta \aleph^{(2)},$ when $d = 3$.}
\label{deltan2}
\end{figure}

For characterizations of hidden regular variation
\citep{maulik:resnick:2005}
 it is useful to 
know if 
$\nu^{(l)}(\{ {\bf{x}} \in \E^{(l)}: ||{\bf{x}}|| > 1 \})$ is finite
or not, where $||{\bf{x}}||$ is any norm of ${\bf{x}}$. Such knowledge
is also useful for estimating probabilities of some risk sets.
For example, if $\nu^{(l)}(\{ {\bf{x}}
\in \E^{(l)}: ||{\bf{x}}|| > 1 \})$ is finite, then so is
$\nu^{(l)}(
\{ {\bf{x}} \in \E^{(l)}: a_1x^1 + a_2x^2 + \cdots + a_dx^d > y
\}, a_i > 0, i= 1, 2, \cdots, d, y > 0$.
We show that this issue can be
resolved by checking a moment condition. 

\subsection{Outline}
Section \ref{notations} explains  notation. In Section
\ref{hiddene02sec}, we review the definitions of  regular
variation on $\E$ and hidden regular variation on $\E^{(2)}$, and
extend the concept to the
sub-cones $\E^{(l)} = [0, \infty]^d \backslash \cup_{1 \le j_1< j_2 <
  \cdots < j_{d-l+1} \le d} \{x^{j_1} = 0, \cdots, x^{j_{d-l+1}} = 0
\}$,  $3 \le l \le d.$ Section \ref{extn_hidden_reg_var}
discusses exploratory detection techniques for hidden regular
variation on $\E^{(l)}$ and estimation of the limit measure $\nu^{(l)}.$
We consider in Section \ref{diff_para_sec} a
transformation that allows us to visualize the hidden angular measure
$S^{(l)}$ through another probability measure $\tilde S^{(l)}$ on the  $(d-1)$-dimensional simplex. In Section
\ref{decide_finite_sec}, we discuss conditions for $\nu^{(l)}(\{
{\bf{x}} \in \E^{(l)}: ||{\bf{x}}|| > 1 \})$ being finite or not. Section \ref{risksetcomputation} gives examples of risk sets where
hidden regular variation helps in obtaining finer estimates of their
probabilities. Our methodologies are applied to  two examples in
Section \ref{sec:eg}. We conclude with some remarks and outline open issues in 
Section \ref{conclusion}.

\subsection{Notation}\label{notations}

\subsubsection{{\bf{Vectors and cones}}} For  denoting a vector and its components, we use:
$${\bf{x}} = (x^1, x^2, \cdots, x^d), \hskip 1 cm x^i = \hbox{$i$-th component of } {\bf{x}}, \hskip 0.1 cm i = 1, 2, \cdots, d.$$
The vectors of all zeros, all ones and all infinities are denoted by
${\bf{0}}= (0, 0, \cdots, 0),$ ${\bf{1}} = (1, 1, \cdots, 1)$ and
${\boldsymbol{\infty}} = (\infty, \infty, \cdots, \infty)$
respectively. Operations on and between vectors are understood
componentwise. In particular, for non-negative vectors $\bx$ and 
 ${\boldsymbol{\beta}} = (\beta^1, \beta^2, \cdots, \beta^d)$, write
 ${\bf{x}}^{{\boldsymbol{\beta}}} = ( {(x^1)}^{\beta^1}, {(x^2)}^{\beta^2}, \cdots, {(x^d)}^{\beta^d})$.
We denote the norm of ${\bf{x}}$ as $|| {\bf{x}}||$. Unless specified, this could be taken as any norm. For the $i$-th largest component of ${\bf{x}}$, we use:
$$ x^{(i)} = \hbox{$i$-th largest component of } {\bf{x}}, \hskip 0.1 cm i = 1, 2, \cdots, d, \hskip 0.3 cm \hbox{i.e.}  \hskip 0.3 cm x^{(1)} \ge x^{(2)} \ge \cdots \ge x^{(d)}. $$
So, the superscripts denote components of a vector and the ordered
component is denoted by a parenthesis in the superscript. 

Sometimes, we have to sort the $i$-th largest components of the
vectors ${\bf{Z}}_1, {\bf{Z}}_2,
\cdots, {\bf{Z}}_n$ in non-increasing order. We
first obtain the vector $\{ Z_1^{(i)}, Z_2^{(i)},
\cdots, Z_n^{(i)} \}$ by taking the $i$-th largest component for each
${\bf{Z}}_j$ and then sort these to get 
$$ Z_{(1)}^{(i)} \ge Z_{(2)}^{(i)} \ge \cdots \ge Z_{(n)}^{(i)}.$$
We use the parentheses in the subscript to avoid double parentheses on
the superscript.

The cones we consider are
\begin{align*}
\E =&  \E^{(1)} =[0, \infty]^d \backslash \{ {\bf{0}} \} = [0,
\infty]^d \backslash \{x^1= 0, \cdots, x^d = 0 \}\\ 
=& [0, \infty]^d
\backslash \{x^{(1)} = 0 \}
=\{\bx \in [0,\infty]^d: x^{(1)}>0\}
\end{align*}
and for $2 \le l \le d$,
\begin{align*}
\E^{(l)} =& [0, \infty]^d \backslash \cup_{1 \le j_1< j_2 < \cdots <
  j_{d-l+1} \le d} \{x^{j_1} = 0, \cdots, x^{j_{d-l+1}} = 0 \} \\
=& [0,
\infty]^d \backslash \{ x^{(l)} = 0 \}
=\{\bx \in [0,\infty]^d: x^{(l)}>0\}.
\end{align*}
For $2 \le l \le d,$ $\E^{(l)}$ is the set of points in $\E$ such that at least $l$ components are positive. Sometimes $\E^{(2)}$ is expressed as $\E^{(2)} = \E \backslash \cup_{i=1}^d \mathbb{L}_i$,
where $\mathbb{L}_i:= \{ t {\bf{e}}_i, t > 0 \}$ is the $i$-th axis
and ${\bf{e}}_i = (0, \ldots, 0, 1, 0, \ldots, 0)$, where $1$ is in
the $i$-th position, $i = 1, 2, \ldots, d$. For ${\bf{x}} \in \E,$ we
use $[{\bf{0, x}}]^c$ to mean 
$
[{\bf{0, x}}]^c = \E \backslash [{\bf{0, x}}] = \{ {\bf{y}} \in \E: \vee_{i = 1}^d {y^i}/{x^i} > 1 \}.
$

\subsubsection{{\bf{Regular variation and vague convergence.}}}
\label{subsec:rv_vague}
 We express  vague convergence \citep[page 173]{resnickbook:2007} of Radon measures as
$\stackrel{v}{\rightarrow}$ and
weak convergence of probability measures \cite[page
14]{billingsley:1999}
as $\Rightarrow$. Denote the set of non-negative Radon
measures on a space $\F$ as $M_+(\F)$ and the set of all non-negative
continuous functions with compact support from $\F$ to $\R^+$ as
$C_K^+(\F)$. The notation $RV_{\rho}$ means the family
of one dimensional regularly varying functions with exponent of
variation $\rho$ (\citep[page 24]{resnickbook:2007},
\cite{bingham:goldie:teugels:1987, 
dehaan:ferreira:2006}). For any
measure $m$ and a real-valued function $f$, denote the integral
$\int f({\bf{x}})m(dx)$ by $m(f)$. 

For defining regular variation of distributions of random vectors on
$\E = \E^{(1)}$ as in \eqref{reg_var_on_general_cone}, we use the
scaling function $b(t) = b^{(1)}(t)$ and get the limit measure $\nu =
\nu^{(1)}$. Similarly, for defining regular variation of
distributions of random vectors on $\E^{(l)}$, $2 \le l \le d$, we use
the scaling function $b^{(l)}(t)$ and get the limit measure
$\nu^{(l)}$. For each $1 \le l \le d$, define the set
$\aleph^{(l)}$ by $\aleph^{(l)}  = \{ {\bf{x}} \in \E^{(l)} : x^{(l)}
\ge 1\}$. 
Since $\aleph^{(l)}$ is compact in $\E^{(l)}$,  there always exists a suitable choice of the scaling
function $b^{(l)}(t)$ which makes $\nu^{(l)}( \aleph^{(l)}) = 1.$  We
assume this from now on.

For each $2 \le l \le d$, if we have hidden regular variation on
$\E^{(l)}$, the limit measure $\nu^{(l)}$  can be expressed in a
convenient coordinate system
as a
product of a Pareto measure 
$\nu_{\alpha^{(l)}} (dr)= \alpha^{(l)}  r^{-\alpha^{(l)}-1}dr,\, r>0
$ and a probability
measure $S^{(l)}$ 
on the compact set $\delta \aleph^{(l)}  = \{ {\bf{x}} \in \E^{(l)} : x^{(l)} = 1\}$.
The  measure $S^{(l)}$ is called the
hidden angular or hidden spectral measure on $\E^{(l)}.$ 
 Whenever
$\nu^{(l)}( \{ {\bf{x}} \in \E^{(l)} : x^{(l)} \ge 1, x^{(1)} = \infty
\}) = 0,$ we view $S^{(l)}$ through its transformed version denoted
$\tilde S^{(l)}$, which is a probability measure on the
$(d-1)$-dimensional simplex $\Delta_{d-1} = \{ {\bf{x}} \in
[0,1]^{d-1} : \sum_{i=1}^{d-1} x^i \le 1 \}$. 

\subsubsection{{\bf{Anti-ranks.}}}\label{subsec:anti}
Suppose, ${\bf{Z}}_1, {\bf{Z}}_2, \cdots, {\bf{Z}}_n$ are random vectors in ${[0,\infty)}^d$. For $j = 1, 2, \ldots, d$, $i = 1, 2, \cdots, n$, define the anti-rank
\begin{equation*}
r_i^j = \sum_{l = 1}^n 1_{\{ Z_l^j \ge  Z_i^j\}}
\end{equation*}
 for $Z_i^j$ to be the number of $j$-th components greater than or equal to $Z_i^j$. For $2 \le l \le d$, define
 \begin{align*}
 m_i^{(l)} = \hbox{the $l$-th largest component of} \hskip 0.2 cm ( \frac{1}{r_i^j}, j = 1, 2, \ldots, d)
 \end{align*}
 and then order them as
 \begin{align*}
 m_{(1)}^{(l)} \ge m_{(2)}^{(l)} \ge \cdots \ge m_{(n)}^{(l)}.
 \end{align*}

\section{Hidden regular variation}\label{hiddene02sec}
We give more details about regular variation
on $\E$ and HRV on $\E^{(2)}$ and then extend the
definitions to hidden regular variation on sub-cones of
$\E^{(2)}$. We illustrate with some examples.

\subsection{Hidden regular variation on $\E^{(2)}$ } Consider regular variation on $\E$ and hidden regular variation
on $\E^{(2)}$. 

\subsubsection{{{The standard case}}}
The distribution of  ${\bf{Z}} = (Z^1, Z^2, \cdots,
Z^d)$ is  regularly varying on $\E:=[0, \infty]^d
\backslash \{{\bf{0}}\}$ with limit measure $\nu$ if there exist a
function $b(t) \uparrow \infty$ as $t \to \infty$ and a non-negative
non-degenerate Radon measure $\nu \ne 0$ such that 
\begin{equation}\label{stan_reg_var}
tP\left[ \frac{{\bf{Z}}}{b(t)} \in \cdot \right] \stackrel{v}{\rightarrow} \nu(\cdot) \quad \text{in $M_+(\E)$}.
\end{equation}
 The limit measure $\nu$ must have all non-zero marginals. Then, there exists $\alpha > 0$ such that $b(\cdot) \in RV_{1/\alpha}$ and $\nu$ satisfies the scaling property
\begin{equation}\label{scaling_e}
\nu (c \cdot) = c^{-\alpha}\nu(\cdot), \hskip 1 cm c > 0.
\end{equation}
Call the limit relation \eqref{stan_reg_var} the standard
case which requires the same scaling function $b(t)$
 for all the components of ${\bf{Z}}$ in
\eqref{stan_reg_var} and ensures that $\nu$ has all non-zero
marginals. 

HRV allows for another regular variation on a
sub-cone such as $\E^{(2)}$.
The distribution of  ${\bf{Z}}$ has  hidden regular variation on
$\E^{(2)}$ if in addition to \eqref{stan_reg_var} there exist a
non-decreasing function $b^{(2)}(t) \uparrow \infty$ such that $b(t)/
b^{(2)}(t) \to \infty$ and a non-negative Radon measure $\nu^{(2)}
\ne 0$ on $\E^{(2)}$ such that 
\begin{equation}\label{hidden_reg_var}
tP\left[ \frac{{\bf{Z}}}{b^{(2)}(t)} \in \cdot \right] \stackrel{v}{\rightarrow} \nu^{(2)}(\cdot) \quad \text{in $M_+(\E^{(2)})$};
\end{equation}
 see \cite[page 324]{resnickbook:2007}. It follows from \eqref{hidden_reg_var} that there exists $\alpha^{(2)} \ge \alpha$ such that $b^{(2)}(\cdot) \in RV_{1/\alpha^{(2)}}$ and $\nu^{(2)}$ satisfies the scaling property
\begin{equation}\label{scaling_e02}
\nu^{(2)} (c \cdot) = c^{-\alpha^{(2)}}\nu^{(2)}(\cdot), \hskip 1 cm c > 0.
\end{equation} 
HRV implies $\nu(\E^{(2)}) =0$, which is known as asymptotic
independence \cite[page 324]{resnickbook:2007}. We emphasize that the
model of hidden regular variation on $\E^{(2)}$ requires both
\eqref{stan_reg_var} and \eqref{hidden_reg_var} to be satisfied with
$b(t)/ b^{(2)}(t) \to \infty,$ and not only regular variation on
$\E^{(2)}$ as in \eqref{hidden_reg_var}. 

\subsubsection{{{The non-standard case}}}

Non-standard regular variation may hold when \eqref{stan_reg_var} fails, but 
\begin{equation}\label{non-stan_reg_var_onE}
tP\left[ \left( \frac{Z^j}{a^j(t)} , j=1, 2, \cdots,d \right) \in \cdot \right] \stackrel{v}{\rightarrow} \mu(\cdot) \quad \text{in $M_+(\E)$}
\end{equation}
for some scaling functions $a^1(\cdot), a^2(\cdot), \cdots,
a^d(\cdot)$ satisfying $a^i(t) \uparrow \infty,$ where $\mu$ is a
non-negative non-zero Radon measure on $\E$ \citep{dehaan:omey:1984,
  resnick:greenwood:1979}. We assume that  marginal convergences
satisfy
\begin{equation}\label{mar_in_non-stan_case}
tP\left[ \frac{Z^j}{a^j(t)} \in \cdot \right]
\stackrel{v}{\rightarrow} \nu_{\beta^j}(\cdot)\quad  \text{in  $M_+((0,\infty])$},
\end{equation}
 where $\nu_{\beta^j}((x, \infty]) = x^{-\beta^j}, \, \beta^j > 0, x > 0.$
Relation \eqref{non-stan_reg_var_onE} is equivalent to 
\begin{equation}\label{non-stan_reg_var_onE_stan}
tP\left[ \left( \frac{{a^j}^{\leftarrow}(Z^j)}{t} , j=1, 2, \cdots,d \right) \in \cdot \right] \stackrel{v}{\rightarrow} \nu(\cdot) \quad \text{in $M_+(\E)$},
\end{equation}
where $\nu$ satisfies the scaling property
$\nu (c \cdot) = c^{-1}\nu(\cdot), \,c > 0,
$
(\cite[page 277]{resnick:1987},
\cite{dehaan:ferreira:2006, heffernan:resnick:2005}). The
limit measures $\nu$ and $\mu$ are related:
\begin{equation}\label{mu_and_nu}
\mu([{\bf{0, x}}]^c) = \nu([{\bf{0,
    x^{{\boldsymbol{\beta}}}}}]^c),\quad {\bf{x}} \in \E.
\end{equation}
In this non-standard case,
the distribution of ${\bf{Z}}$ has hidden regular variation on
$\E^{(2)}$ 
 if, in addition to \eqref{non-stan_reg_var_onE_stan}, there exist a
 non-decreasing function $b^{(2)}(t) \uparrow \infty$, such that
 $t/b^{(2)}(t) \to \infty$, 
and a non-negative non-zero Radon measure $\nu^{(2)}$ on $\E^{(2)}$
satisfying
\begin{equation}\label{non-stan_hidden_reg_var_onE02}
tP\left[ \frac{\left( {a^j}^{\leftarrow}(Z^j), j =1,2, \cdots,
      d\right)}{b^{(2)}(t)} \in \cdot \right]
\stackrel{v}{\rightarrow} \nu^{(2)}(\cdot) \quad  \text{in $M_+(\E^{(2)})$}.
\end{equation} 
 Then, there exists $\alpha^{(2)} \ge 1$ such that $b^{(2)}(\cdot) \in RV_{1/\alpha^{(2)}}$ and $\nu^{(2)}$ satisfies the scaling property \eqref{scaling_e02}.

Note that \eqref{non-stan_reg_var_onE_stan} standardizes
\eqref{non-stan_reg_var_onE} with scaling function $b(t) = t$, and
the definition of hidden regular variation on $\E^{(2)}$
in \eqref{non-stan_hidden_reg_var_onE02}, is the most natural
substitute for \eqref{hidden_reg_var}. This reduces the non-standard case
to the standard one. Of course, we have to deal with the unknown nature of
the scaling functions
$a^j(\cdot), \hskip 0.2 cm j=1, 2, \cdots, d$.

\subsection{Hidden regular variation beyond $\E^{(2)}$}
For dimension $d > 2,$ it is possible to refine the model of HRV on $\E^{(2)}$ by defining hidden regular variation
on sub-cones of $\E^{(2)}.$ For $d > 2,$  even in the absence of
asymptotic independence, it is possible to
define HRV on sub-cones of $\E^{(2)}$ and the family of distributions
satisfying HRV on some sub-cone of
$\E^{(2)}$ is not a subfamily of distributions satisfying 
HRV on $\E^{(2)}$. 

\subsubsection{{{Motivation}}} 
A reason for seeking HRV on $\E^{(2)}$ is that
in the presence of asymptotic independence
when the limit measure $\nu$ puts zero mass on $\E^{(2)},$
 regular variation on $\E$ may fail to provide non-zero
 estimates of the probabilities of remote critical sets 
such as failure regions (reliability), overflow regions (hydrology),
and out-of-compliance regions (environmental protection).
Beyond $\E^{(2)}$, if  the limit measure $\nu^{(2)}$
in \eqref{hidden_reg_var} puts zero mass on $\E^{(3)}$ we would seek
to  refine HRV on
$\E^{(2)}.$

Consider the following thought experiment. Suppose, ${\bf{Z}} = ( Z^1, Z^2, \cdots, Z^d)$ represents
concentrations of a pollutant at $d$ locations and 
that {\bf{Z}} has a regularly varying distribution on $\E$ with
asymptotic independence. Assume
 we found HRV on $\E^{(2)}$ and the limiting measure $\nu^{(2)}$
in this case satisfies $\nu^{(2)}(\E^{(3)}) = 0$, so
HRV on $\E^{(2)},$ estimates $P( Z^{j_1} > x_1,
Z^{j_2} > x_2, \cdots Z^{j_l} > x_l)$  to be $0$ for
$3 \le l \le d$ and $1 \le j_1 < j_2 < \cdots < j_l \le d$. This resulting
estimate seems crude and we seek a remedy by looking for  finer structure of
 on the sub-cones $\E^{(3)} \supset \cdots \supset
\E^{(d)}$ in a sequential manner. 

Another context for HRV on $\E^{(3)}$ is as a  refinement of  
regular variation on $\E$ when asymptotic
independence is absent. Suppose, in the above thought experiment, {\bf{Z}} has a
regularly varying distribution on $\E$ with limit measure $\nu$ such
that $\nu(\E^{(2)}) > 0,$ but $\nu(\E^{(3)}) = 0.$ Asymptotic
independence is absent, but $P( Z^{j_1} > x_1, Z^{j_2} >
x_2, \cdots Z^{j_l} > x_l)$ is estimated to be $0$ for all $3 \le l
\le d$ and $1 \le j_1 < j_2 < \cdots < j_l \le d$. This suggests seeking
HRV on the sub-cones $\E^{(3)} \supset \cdots
\supset \E^{(d)}.$  

Examples in Section \ref{hidden_beyond_examples_section} show
 each modeling situation we considered in the above  thought
 experiments can happen.

We seek regular variation on the cones $\E  \supset \E^{(2)}
\supset \E^{(3)} \supset \cdots \supset \E^{(d)}$ in a sequential
manner. If  for some $1 \le j  \le d$, regular
variation is present on $\E^{(j)}$, as in
\eqref{reg_var_on_general_cone} and the limit measure $\nu^{(j)}$ puts
non-zero mass on $\E^{(l)}$, $j<l\leq d$, i.e. $\nu^{(j)}(\E^{(l)}) > 0$, then
there is no need to seek HRV on any of the cones
$\E^{(j+1)} \supset \cdots \supset \E^{(l)}$.
Recall the conventions that we replace  $\nu$, $\alpha$, $\E$
and $b(t)$ by
 $\nu^{(1)}$, $\alpha^{(1)}$, $\E^{(1)}$ and $b^{(1)}(t)$
 respectively.

Of course, there are other ways to nest sub-regions of $\E$ and seek
regular variation but our sequential search for regular variation on the cones
$\E^{(l)};\,l=2,\dots,d$ is one structured approach to the problem of
refined estimates.

\subsubsection{{{Formal definition of HRV on $\E^{(l)}$}} } \label{formal_definition_model}
The definition proceeds sequentially and begins with the standard case. Assume that ${\bf{Z}}$
satisfies  regular variation on $\E^{(1)}$ as in
\eqref{stan_reg_var} and that we have  regular
variation on a sub-cone $\E^{(j)} $ with scaling
function $b^{(j)}(t) \in RV_{1/\alpha^{(j)}}$ and limiting Radon measure $\nu^{(j)} \ne 0$.
For $j<l\leq d$, further assume that $\nu^{(j)}(\E^{(l-1)}) > 0$ and
$\nu^{(j)}(\E^{(l)}) = 0$.
 The cone $\E^{(j)}$ could be $\E^{(1)}$. The distribution
   of ${\bf{Z}}$ has hidden regular variation on 
 $\E^{(l)}$, if in addition to regular variation on $\E^{(j)}$, there
 is a  non-decreasing function $b^{(l)}(t) \uparrow \infty$ such that
 $b^{(j)}(t)/ b^{(l)}(t) \rightarrow \infty$, and a non-negative Radon
 measure $\nu^{(l)} \ne 0$ on $\E^{(l)}$ such that
\begin{equation}\label{reg_var_on_E0l}
tP\left[ \frac{{\bf{Z}}}{b^{(l)}(t)} \in \cdot \right] \stackrel{v}{\rightarrow} \nu^{(l)}(\cdot) \quad \text{in $M_+(\E^{(l)})$}.
\end{equation}
{}From \eqref{reg_var_on_E0l}, there exists $\alpha^{(l)} \ge \alpha^{(j)}$ such that $b^{(l)}(\cdot) \in RV_{1/\alpha^{(l)}}$ and $\nu^{(l)}$ has the scaling property 
\begin{equation}\label{scaling_e0l}
\nu^{(l)}(c \cdot) = c^{-\alpha^{(l)}}\nu^{(l)}( \cdot), \hskip 1 cm c > 0.
\end{equation}

For vague convergence on $\E^{(l)}$, it is important to
identify the compact sets of $\E^{(l)}$. From Proposition
6.1 of 
\cite[page 171]{resnickbook:2007}, the compact sets of $\E^{(l)}$ are
 closed sets contained in sets of the form $\{ {\bf{x}}
\in \E^{(1)} : x^{j_1} > w_1, x^{j_2} > w_2, \cdots, x^{j_l} > w_l \}$
for some $ 1 \le j_1 < j_2 < \cdots < j_l \le d$ and for some  $w_1,
w_2, \cdots, w_l > 0$. So, for all $\delta > 0$, $[x^{(l)}>\delta] $
is compact and 
\begin{equation*}
tP\left[ \frac{{\bf{Z}}}{b^{(j)}(t)} \in \{ {\bf{x}} \in \E^{(1)} : x^{(l)} > \delta \} \right] = tP\left[ \frac{Z^{(l)}}{b^{(l)}(t)}  > \frac{b^{(j)}(t)}{b^{(l)}(t)}\delta \right] \rightarrow 0,
\end{equation*}
since $b^{(j)}(t)/b^{(l)}(t) \to \infty$. Therefore, 
$\nu^{(j)}(\E^{(l)}) = 0$ is a necessary condition for HRV on $\E^{(l)}$.

 For defining HRV in the non-standard case, assume
\eqref{non-stan_reg_var_onE_stan} holds on $\E^{(1)}$ and the
rest of the definition
is the same with
 ${\bf{Z}}$ and $b^{(1)}(t)$  replaced by $({a^1}^{\leftarrow}(Z^1), {a^2}^{\leftarrow}(Z^2), \cdots, {a^d}^{\leftarrow}(Z^d))$ and $t$ respectively.

\begin{rem}
\rm{A few important remarks about hidden regular variation:

\begin{enumerate}[(i)]
\item The definition of hidden regular variation  leading to \eqref{reg_var_on_E0l}
is consistent with the definition of hidden regular variation on $\E^{(2)}.$ 

\item The definition of regular variation on $\E^{(1)}$ as in
  \eqref{stan_reg_var} or \eqref{non-stan_reg_var_onE_stan} requires
  that the limit measure $\nu^{(1)}$ has  non-zero
  marginals. When defining regular variation on $\E^{(l)},
  \,2 \le l \le d,$ as in \eqref{reg_var_on_E0l}, we
 do  not demand such a condition. For instance,  ${\bf{Z}} = (Z^1, Z^2, Z^3)$ being
  regularly varying on $\E^{(2)}$ does not imply that $(Z^1, Z^2)$ is
  regularly varying on $(0,  \infty]^2.$ See Example
  \ref{at_e02_but_not_at_e03}. 
  
\item Non-standard regular variation
    allows each component $Z^j$ of the random vector $\bf{Z}$ to be scaled by a
  possibly different scaling function $a^j(t)$ as in \eqref
  {non-stan_reg_var_onE}. An alternative approach to 
 defining regular
  variation on $\E^{(l)}$, $2 \le l \le d$,  would allow each component
  $Z^j$ of the random vector $\bf{Z}$ to be scaled by a possibly
  different scaling function $b^{(l),j}(t)$ and this would
    produce a more general
  model of HRV than the one we defined. However, we do not have a method of
  estimating the scaling functions $b^{(l),j}(n/k)$; see estimation of $b^{(l)}(n/k)$ using
  \eqref{estimatorb2n/k} and estimation of $a^j(b^{(2)}(n/k))$ in the
  non-standard case using
  \eqref{non-stan_estimatorbjn/k}. 
\end{enumerate}
}
\end{rem}

\subsection{Examples}\label{hidden_beyond_examples_section}
We give examples to exhibit subtleties. Example
\ref{no_at_e02_but_at_e03} shows a model in which
HRV is
not present in $\E^{(2)}$ but is present in $\E^{(3)}$. So,
non-existence of HRV on $\E^{(2)}$ does not
preclude  HRV on $\E^{(3)}.$
Example \ref{no_at_e02_but_at_e03} also shows that asymptotic
independence is not a necessary condition for the presence of HRV on $\E^{(3)}.$ In Examples
\ref{at_e02_but_not_at_e03} and \ref{e02_e03}, we  learn that HRV on $\E^{(2)}$ does not imply HRV on $\E^{(3)}.$ In Example \ref{e02_e03}, HRV on
$\E^{(3)}$ fails
 because $\nu^{(2)}(\E^{(3)}) >
0,$ but a different reason for failure holds in
 Example \ref{at_e02_but_not_at_e03}. 
In contrast,  Example \ref{iid_pareto} demonstrates that HRV could be present on each of the sub-cones
$\E^{(l)}$, $2 \le l \le d$. Also, Example \ref{e02_e03} shows that
asymptotic independence, unlike independence, does not imply $\nu^{(2)}( \E^{(3)}) = 0.$  
\begin{ex}\label{iid_pareto}
\rm{ An extension of Example 5.1 of \cite{maulik:resnick:2005}: Suppose, $Z^1, Z^2, \cdots, Z^d$ are iid Pareto($1$). Then, regular variation of ${\bf{Z}} = (Z^1, Z^2, \cdots, Z^d)$ is present on $\E$ with $\alpha =1$ and HRV is present on each of the sub-cones $\E^{(l)}$ with $\alpha^{(l)} = l$, for $2 \le l \le d$.
}
\end{ex}
\begin{ex}\label{no_at_e02_but_at_e03}
\rm{Suppose, $X$ and $Y$ are iid Pareto($1$) and ${\bf{Z}} =
  (X,2X,Y)$, so 
\begin{equation*}
tP\left[ \frac{{\bf{Z}}}{2t} \in \cdot \right] \stackrel{v}{\rightarrow} \nu(\cdot) \quad \text{in $M_+(\E)$},
\end{equation*}
and $\nu$ has all non-zero marginals.  However,  ${\bf{Z}}$ does not possess asymptotic
independence
 since $Z^1$ and $Z^2$ are not asymptotically independent
\citep[page 296, Proposition 5.27]{resnick:1987} and thus HRV cannot be present on $\E^{(2)}$ \citep[page 325, Property
9.1]{resnickbook:2007}. However,
\begin{align*}
\nu(\E^{(3)}) &= \lim_{w \to 0} \nu( \{ {\bf{x}} : x^1 \wedge x^2 \wedge x^3 > w \}) = \lim_{w \to 0} \lim_{t \to \infty} t P\left[ X > tw, 2X > tw, Y> tw \right] \\
&= \lim_{w \to 0} \lim_{t \to \infty} t P\left[ X > tw, Y> tw \right] = \lim_{w \to 0} \lim_{t \to \infty} t {(tw)}^{-1}{(tw)}^{-1} = 0.
\end{align*}
This suggests seeking HRV on $\E^{(3)}$ and
indeed this holds with $b^{(3)}(t) = \sqrt{t}$ since
 for  $w_1, w_2, w_3 > 0$,
\begin{align*}
\lim_{t \to \infty} t P\left[ X > \sqrt{t}w_1, \right. & \left. 2X > \sqrt{t}w_2, Y> \sqrt{t}w_3 \right] = \lim_{t \to \infty} t P\left[ X > \sqrt{t} \left(w_1 \vee \frac{w_2}{2}\right), Y> \sqrt{t}w_3 \right]\\
&= \lim_{t \to \infty} t {\left[\sqrt{t} \left(w_1 \vee \frac{w_2}{2}\right)\right]}^{-1}{(\sqrt{t}w_3)}^{-1} = \frac{1}{\left(w_1 \vee \frac{w_2}{2}\right)w_3}.
\end{align*}

So, for this example, 
\begin{enumerate}[(i)]
\item Regular variation holds on $\E^{(1)}$ and $\E^{(2)}$ (since $\nu(\E^{(2)}) \ne 0$), HRV holds on $\E^{(3)}, \hskip 0.1 cm  \nu(\E^{(1)}) = \nu(\E^{(2)}) = \infty, \hskip 0.1 cm \nu(\E^{(3)}) = 0.$ 
\item Asymptotic independence is absent but HRV on $\E^{(3)}$ is present. 
\end{enumerate}
}
\end{ex}

\begin{ex}\label{at_e02_but_not_at_e03} \rm{
Example 5.2 from \cite{maulik:resnick:2005}: Let, $X_1, X_2, X_3$ be
iid Pareto($1$) random variables. Also, let $B_1, B_2$ be iid
Bernoulli random variables independent of $(X_1,X_2,X_3)$ with
$
P[ B_i = 1] = P[ B_i = 0 ] = 1/2,\, i = 1, 2. 
$
Define ${\bf{Z}}=( B_2 X_1,(1- B_2)X_2, (1 - B_1)X_3)$.
From \cite{maulik:resnick:2005},
 HRV exists on the cone $\E^{(2)}$ with
 $\alpha^{(2)} =2$ and $\nu^{(2)}$ concentrates on $[x^1 > 0, x^3 > 0]
 \cup [x^2 > 0, x^3 > 0].$ Also,  $\nu^{(2)}(\{ {\bf{x}} :  x^1 > 0,
 x^2 > 0 \} ) = 0$. Since, $\E^{(3)}$ is a subset of $\{ {\bf{x}} :
 x^1 > 0, x^2 > 0 \}$, $\nu^{(2)}(\E^{(3)}) = 0$. However, HRV on $\E^{(3)}$ fails. The compact sets of $\E^{(3)}$
 are contained in sets of the form $\{ {\bf{x}} : x^1 > w^1, x^2 >
 w^2, x^3 > w^3 \}$ for $w^1, w^2, w^3 > 0$. Since either $Z^1$
 or $Z^2$ must be zero, for any increasing function $h(t) \uparrow
 \infty$, and l $w^1, w^2, w^3 > 0$, we have
\begin{equation*}
\lim_{t \to \infty} tP\left[ \frac{{\bf{Z}}}{h(t)} \in \{ {\bf{x}} : x^1 > w^1, x^2 > w^2, x^3 > w^3 \} \right] = 0.
\end{equation*}
 Hence, HRV holds on $\E^{(2)}$
with $ b^{(2)}(t) = \sqrt{t}$, but HRV on $\E^{(3)}$ fails.
}
\end{ex}

\begin{ex}\label{e02_e03}
\rm{ Let $X_1$, $X_2$ and $X_3$ be iid Pareto(1) random variables
 and define ${\bf{Z}} = \bigl(
(X_1)^2 \wedge (X_2)^2,
 (X_2)^2 \wedge (X_3)^2,
(X_1)^2 \wedge   (X_3)^2\bigr)$. First, note that 
\begin{align*}
tP\left[ \frac{{\bf{Z}}}{3t} \in \cdot \right] \stackrel{v}{\rightarrow} \nu(\cdot) \quad \text{in $M_+(\E)$}
\end{align*}
for some non-zero Radon measure $\nu$ on $\E$ with non-zero marginals. Also, 
\begin{align*}
tP\left[ \frac{{\bf{Z}}}{t^{2/3}} \in \cdot \right] \stackrel{v}{\rightarrow} \nu^{(2)}(\cdot) \quad\text{in $M_+(\E^{(2)})$}
\end{align*}
 for a non-zero Radon measure $\nu^{(2)}$ on $\E^{(2)}$. So, HRV exists on $\E^{(2)}$ and hence, the components of ${\bf{Z}}$ are asymptotically independent \citep[page 325, Property 9.1]{resnickbook:2007}. For $w_1, w_2, w_3 > 0$,
\begin{align*}
\lim_{t \to \infty}tP&\left[ \frac{{\bf{Z}}}{t^{2/3}} \in \{ {\bf{x}} : x^1 > w_1, x^2 > w_2, x^3 > w_3 \}\right] \\
& \qquad = \lim_{t \to \infty} tP\left[ X_1 > t^{1/3}{(w_1\vee w_3)}^{1/2}, X_2 > t^{1/3}{(w_1\vee w_2)}^{1/2}, X_3 > t^{1/3} {(w_2 \vee w_3)}^{1/2} \right] \\
&\qquad = \frac{1}{\sqrt{(w_1\vee w_3)\cdot (w_1\vee w_2)\cdot (w_2\vee w_3)}} = \nu^{(2)}( \{ {\bf{x}} : x^1 > w_1, x^2 > w_2, x^3 > w_3 \}).
\end{align*}
As $\{ {\bf{x}} : x^1 > w_1, x^2 > w_2, x^3 > w_3 \} \subset \E^{(3)}$, $\nu^{(2)}(\E^{(3)}) > 0$. So, for this example, 
\begin{enumerate}[(i)]
\item HRV exists on $\E^{(2)},$ not on
  $\E^{(3)},$ but ${\bf{Z}}$ is regularly varying on $\E^{(3)}$ in the sense of \eqref{reg_var_on_general_cone}. 
\item Asymptotic independence holds but  $\nu^{(2)}(\E^{(3)}) > 0.$
\end{enumerate}
}
\end{ex}

\section{Exploratory detection and estimation techniques} \label{extn_hidden_reg_var}
Existing exploratory detection techniques for HRV on
$\E^{(2)}$ are valid in two dimensions.
Our methods, applicable to any dimension, also allow for sequential
search for HRV
on $\E^{(l)}, \hskip 0.1 cm 2 \le l \le d.$ 

We find a coordinate system in which the limit measure
$\nu^{(l)}$ in \eqref{reg_var_on_E0l} is a product of  a probability
measure and a Pareto measure of the form $\nu_{\alpha^{(l)}}$ for some
$\alpha^{(l)} > 0.$ Thus we exploit the semi-parametric nature of 
$\nu^{(l)}$ for estimation and detection.

\subsection{ Decomposition of the limit measure $\nu^{(l)}$ }\label{hidden_ang_meas_sec}
By a suitable choice of  scaling function $b^{(l)}(t)$, we can
  and do make $\nu^{(l)}( \aleph^{(l)}) = 1$, where $\aleph^{(l)}  =
\{ {\bf{x}} \in \E^{(l)} : x^{(l)} \ge 1\}$. We decompose $\nu^{(l)}$
into a Pareto measure
$\nu_{\alpha^{(l)}}$ and 
 a probability measure $S^{(l)}$ on $\delta \aleph^{(l)}  = \{
{\bf{x}} \in \E^{(l)} : x^{(l)} = 1\}$ called the  hidden spectral or hidden angular measure.

\begin{prop}\label{stan_transform}
\rm{The distribution of the random vector ${\bf{Z}}$ has regular variation on $\E^{(l)},$ i.e. it satisfies \eqref{reg_var_on_general_cone} with $\mathbb{C} = \E^{(l)}$ and $\chi = \nu^{(l)}$, and the condition $\nu^{(l)}(\aleph^{(l)}) = 1$ holds iff 
\begin{equation}\label{stan_case_transform}
tP\left[ \left( \frac{Z^{(l)}}{b^{(l)}(t)}, \frac{{\bf{Z}}}{Z^{(l)}} \right) \in \cdot \right] \stackrel{v}{\rightarrow} \nu_{\alpha^{(l)}} \times S^{(l)}(\cdot)  \quad \text{in $M_+((0, \infty] \times \delta \aleph^{(l)})$},
\end{equation}
 where $Z^{(l)}$ is the $l$-th largest component of       
 ${\bf{Z}}$. The limit measure $\nu^{(l)}$ and the probability measure $S^{(l)}$ are related by 
  \begin{equation}\label{nu0ls0l}
 \nu^{(l)}(\{ {\bf{x}} \in \E^{(l)} : x^{(l)} \ge r, \frac{{\bf{x}}}{x^{(l)}} \in \Lambda \}) = r^{-\alpha^{(l)}} S^{(l)}(\Lambda),
 \end{equation}
which holds for all $r > 0$ and all Borel sets $\Lambda \subset \delta \aleph^{(l)}$.
 }
\end{prop}

\begin{proof} See Appendix \ref{proof_stan_transform}.
\end{proof}

\begin{rem}
\rm{ Proposition \ref{stan_transform} only assumes regular variation
  on $\E^{(l)}$, $\alpha^{(l)} > 0,$ whereas  hidden
  regular variation on $\E^{(l)}$ also requires \eqref{stan_reg_var}
  to hold and $b(t)/b^{(l)}(t) \to \infty$. 

Also, the convergence in \eqref{stan_case_transform}  is equivalent to 
\begin{enumerate}[(i)]
\item $Z^{(l)}$ having regularly varying tail with index $\alpha^{(l)} > 0$ 
and
\item as $t \to \infty,$ 
\begin{equation*}
P\left[ \frac{{\bf{Z}}}{Z^{(l)}} \in \cdot \hskip 0.1 cm {\Big{|}} Z^{(l)} > t \right] \Rightarrow S^{(l)}(\cdot) \quad \text{on $\delta \aleph^{(l)}$}.
\end{equation*}

\end{enumerate}
}
\end{rem}

\begin{rem}
\rm{The  polar coordinate transformation ${\bf{x}}
  \mapsto \left( ||{\bf{x}}||, {{\bf{x}}}/{||{\bf{x}}||} \right)$
usually used for regular variation
  introduces a non-compact unit sphere $\{ {\bf{x}} \in \E^{(l)} :
  ||{\bf{x}}|| = 1 \}.$ This defect is fixed by using  $\delta
  \aleph^{(l)}$ instead.} 
\end{rem}

Example \ref{simulation_e0l}
uses Proposition \ref{stan_transform} to construct random variables having regular variation on the cone $\E^{(l)}$ with the limit measure $\nu^{(l)}$.

\begin{ex}\label{simulation_e0l}
\rm{Suppose, $(R, {\bf{\Theta}})$ is an independent pair of random variables on $(0, \infty] \times \delta \aleph^{(l)}$ with
$$ P[ R > r] = r^{-\alpha^{(l)}}, \hskip 0.2 cm r >1, \hskip 1 cm P[ {\bf{\Theta}} \in \cdot] = S^{(l)}(\cdot). $$
Then,
\begin{equation*}
tP\left[ \frac{R}{t^{1/\alpha^{(l)}}} > r, {\bf{\Theta}} \in \Lambda \right] = t{\left(t^{1/\alpha^{(l)}}r\right)}^{-\alpha^{(l)}} S^{(l)}(\Lambda) = r^{-\alpha^{(l)}} S^{(l)}(\Lambda).
\end{equation*}
 By Proposition \ref{stan_transform}, the distribution of ${\bf{Z}} = R{\bf{\Theta}}$ 
is regularly varying 
 on $\E^{(l)}$ and satisfies \eqref{reg_var_on_E0l} with $\nu^{(l)}(\aleph^{(l)}) = 1$. This, however, does not guarantee regular variation on $\E$. Also, unless ${\bf{\Theta}}$ has a support contained in $\{ {\bf{\theta}} \in \delta \aleph^{(l)} : \theta^{(1)} < \infty \}$, the random variable ${\bf{Z}}$ might not be real-valued.
 }
\end{ex}

\subsection{Detection of HRV on $\E^{(l)}$ and estimation of $\nu^{(l)}$} \label{detect}
Is the model of hidden regular variation  on $\E^{(l)}$ appropriate
for a given  data set?
 If so, how do we  estimate the limit measure $\nu^{(l)}$
and tail probabilities of the
 form $P[ Z^{i_1} > z_1, Z^{i_2} > z_2, \cdots Z^{i_l} > z_l ]$ for
  $1 \le i_1 < i_2 \cdots < i_l \le d.$ 
We consider the standard and  non-standard cases and
assume $\nu^{(l)}(\aleph^{(l)}) = 1$. 

\subsubsection{{{The standard case}}}\label{E0l_stan_est_proc_sec}
Suppose, ${\bf{Z}}_1, {\bf{Z}}_2, \cdots, {\bf{Z}}_n$ are iid random
vectors in ${[0,\infty)}^d$ whose common distribution
 satisfies regular variation
on $\E$ as in \eqref{stan_reg_var}. We want to detect if HRV
is present in $\E^{(l)}$ and this requires prior detection of
regular variation on a bigger sub-cone $\E^{(j)} \supset \E^{(l)}$
with the limit measure $\nu^{(j)}$ having
the property $\nu^{(j)}(\E^{(l-1)}) > 0$ and $\nu^{(j)}(\E^{(l)}) =
0.$ 
Recall $\E^{(j)}$ could  be $\E^{(1)}$. 

Here is a method for verifying that
 $\nu^{(j)}(\E^{(l-1)}) > 0$ and $\nu^{(j)}(\E^{(l)}) =
0.$ For each $p > j,$ define a transformation $M^{(p)} : \delta
\aleph^{(j)} \mapsto [0, 1]$ as ${\bf{x}} \mapsto x^{(p)}$. If  $\nu^{(j)}(\E^{(l-1)}) > 0$ and $\nu^{(j)}(\E^{(l)}) = 
0,$ then the probability measure $S^{(j)} \circ {M^{(l)}}^{-1}$
is degenerate at zero but  $S^{(j)} \circ
{M^{(l -1)}}^{-1}$ is not; see Remark
\ref{rem_lemma_2}. As will be discussed later, we can construct an
atomic measure $\hat S^{(j)}$, which consistently estimates
$S^{(j)}$. Using the atoms of $\hat S^{(j)} \circ {M^{(l-1)}}^{-1},$
we plot a kernel density estimate of the density of $S^{(j)} \circ
{M^{(l-1)}}^{-1}.$ If the plotted density appears to concentrate
around zero, we believe that $\nu^{(j)}(\E^{(l-1)}) = 
0.$ Otherwise, we assume that $\nu^{(j)}(\E^{(l-1)}) > 0.$ Then, using
similar methods, we proceed to check whether $\nu^{(j)}(\E^{(l)}) =
0.$ 

Once convinced
 that $\nu^{(j)}(\E^{(l-1)}) > 0$ and $\nu^{(j)}(\E^{(l)}) = 0$, we
seek  HRV on
$\E^{(l)}$. 
Using Proposition \ref{stan_transform}, HRV implies 
\begin{equation}\label{l_component_reg_var_eqn}
tP\left[ {Z^{(l)}}/{b^{(l)}(t)} \in \cdot \right] 
\stackrel{v}{\rightarrow} \nu_{\alpha^{(l)}}(\cdot) \quad \text{in $M_+((0, \infty])$}.
\end{equation}
So, we apply Hill, QQ and Pickands plots to the iid data $\{
Z^{(l)}_i , i = 1, 2, \cdots, n \}$ and attempt to infer that  $Z^{(l)}$
has a regularly varying distribution \citep[Chapter 4]{resnickbook:2007}.  

If convinced that HRV is present, we estimate the limit measure $\nu^{(l)}$.
Define the set 
$$\E_{l\setminus \infty}
 = \E^{(l)} \backslash \cup_{1 \le j_1 < j_2 <
  \cdots < j_l \le d} [ x^{j_1} = \infty, x^{j_2} = \infty, \cdots,
x^{j_l} = \infty] =\E^{(l)}  \setminus [x^{(l)}=\infty]
$$
 and the transformation $Q^{(l)} : \E_{l\setminus \infty} \mapsto (0,
\infty) \times \delta \aleph^{(l)}$ 
as
\begin{equation}\label{define_T_e0l}
Q^{(l)}({\bf{x}}) = \left( x^{(l)}, {{\bf{x}}}/{x^{(l)}} \right).
\end{equation}
 From \eqref{nu0ls0l} and the fact that $Q^{(l)}$ is one-one, we get for any Borel set $A \subset \E^{(l)}$, 
\begin{equation*}
\nu^{(l)}(A) = \nu^{(l)}(A \cap \E_{l\setminus \infty} ) =
\nu_{\alpha^{(l)}} \times S^{(l)}(Q^{(l)}(A \cap \E_{l\setminus
  \infty} )).
\end{equation*}
So, estimating $\alpha^{(l)}$ and the hidden spectral measure
$S^{(l)}$ is equivalent to  estimating  $ \nu^{(l)}$. 

We estimate $\alpha^{(l)}$ using one dimensional methods such as
 the Hill, QQ or  Pickands estimator applied to the iid data $\{
 Z^{(l)}_i , i = 1, 2, \cdots, n \}$. 
An estimator of $S^{(l)}$ can be constructed using standard
ideas as follows \cite{heffernan:resnick:2005}. Suppose, $k(n) \to \infty$, ${k(n)}/{n} \to 0$, as $n \to \infty$. Using Theorem 5.3(ii) of \cite[page 139]{resnickbook:2007}, we get
\begin{equation}\label{almost_spec}
\frac{1}{k} \sum_{i =1}^n \epsilon_{\left( Z^{(l)}_i/b^{(l)}(\frac{n}{k}), {\bf{Z}}_i/ Z^{(l)}_i \right) } \Rightarrow \nu_{\alpha^{(l)}}\times S^{(l)}
\end{equation}
on $M_+((0, \infty] \times \delta \aleph^{(l)} )$.
 Choosing $\left([1, \infty] \times \cdot\right)$ as the set in
 \eqref{almost_spec}, gives an estimator of $S^{(l)}$,
 but this estimator uses the unknown $b^{(l)}(n/k)$, 
which must be replaced by a statistic.

Order the observations $\{ Z^{(l)}_i, i= 1, 2, \cdots, n \}$ as
$Z^{(l)}_{(1)} \ge Z^{(l)}_{(2)} \ge \cdots \ge Z^{(l)}_{(n)}$ which
are order statistics from a sample drawn from a regularly varying distribution. Using \eqref{l_component_reg_var_eqn}  and Theorem 4.2 of \cite[page 81]{resnickbook:2007}, we get
\begin{equation}\label{estimatorb2n/k}
\frac{Z^{(l)}_{(k)}}{ b^{(l)}(n/k)} \stackrel{P}{\rightarrow} 1.
\end{equation}
Then \eqref{almost_spec} and \eqref{estimatorb2n/k} yield
\begin{equation}\label{joint_almostk_estimatorb2n/k}
\left( \frac{1}{k} \sum_{i =1}^n \epsilon_{\left(
      Z^{(l)}_i/b^{(l)}({n}/{k}), {\bf{Z}}_i/ Z^{(l)}_i \right) }
, Z^{(l)}_{(k)}/
  b^{(l)}\left({n}/{k}\right) \right) \Rightarrow
\left(\nu_{\alpha^{(l)}} \times S^{(l)}, 1 \right) 
\end{equation}
on $M_+\bigl((0,\infty] \times \delta \aleph^{(l)} \bigr) \times (0,
\infty]$. Applying the almost surely continuous map  
\begin{equation*}
\left( \nu \times S , x\right) \mapsto  \nu \times S ([x, \infty] \times \cdot )
\end{equation*}
to \eqref{joint_almostk_estimatorb2n/k},  the continuous
mapping theorem \cite[page 21]{billingsley:1999} gives
\begin{align}\label{stan_spec_est_conv_eqn}
 \frac{1}{k} \sum_{i =1}^n \epsilon_{\left( Z^{(l)}_i/b^{(l)}({n}/{k}), {\bf{Z}}_i/ Z^{(l)}_i \right) }&( [Z^{(l)}_{(k)}/ b^{(l)}\left({n}/{k}\right), \infty]  \times \cdot)  \nonumber \\
&=  \frac{1}{k} \sum_{i =1}^n \epsilon_{\left( Z^{(l)}_i/Z^{(l)}_{(k)}, {\bf{Z}}_i/ Z^{(l)}_i \right) }( [1, \infty]  \times \cdot ) \Rightarrow \nu_{\alpha^{(l)}}([1, \infty])S^{(l)}(\cdot) = S^{(l)}(\cdot)
\end{align}
 on $\delta \aleph^{(l)}$. Thus, a consistent estimator
 for $S^{(l)}(\cdot)$ is
$\frac{1}{k} \sum_{i =1}^n \epsilon_{\left( Z^{(l)}_i/Z^{(l)}_{(k)},
    \hskip 0.1 cm {\bf{Z}}_i/ Z^{(l)}_i \right) }( [1, \infty]  \times
\cdot )$ or
\begin{equation}\label{stan_spec_est}
\hat S^{(l)}(\cdot) := \frac{
 \sum_{i =1}^n \epsilon_{\left( Z^{(l)}_i/Z^{(l)}_{(k)}, \hskip 0.1 cm
     {\bf{Z}}_i/ Z^{(l)}_i \right) }( [1, \infty]  \times \cdot )
}
{\sum_{i =1}^n \epsilon_{ Z^{(l)}_i/Z^{(l)}_{(k)} }([1,\infty])
}.
\end{equation}

\subsubsection{{{The non-standard case}}}\label{detect_nonstan_e02}
Suppose, ${\bf{Z}}_1, {\bf{Z}}_2, \ldots {\bf{Z}}_n$ are iid random
vectors in ${[0,\infty)}^d$ such that their common distribution
satisfies non-standard regular variation
\eqref{non-stan_reg_var_onE_stan}
on $\E$. We seek HRV
on $\E^{(l)}$. HRV is defined
sequentially, so if HRV on $\E^{(l)}$ exists,  
\begin{enumerate}[(i)] 
\item either \eqref{non-stan_reg_var_onE_stan} holds, 
and $\left( {a^i}^{\leftarrow}(Z^i),\,i =1,2,
    \cdots, d\right)$ is standard regularly varying on $\E^{(j)},
  \E^{(l)}$ with limit measures $\nu^{(j)}, \nu^{(l)}$ and scaling
  functions $b^{(j)}(t), b^{(l)}(t)$ for $1\leq j<l\leq d$ and $\nu^{(j)}(\E^{(l-1)}) > 0$ and
  $\nu^{(j)}(\E^{(l)}) = 0,$

\item or \eqref{non-stan_reg_var_onE_stan} holds, \eqref{reg_var_on_E0l}
  holds with $\bf{Z}$ replaced by
$\left( {a^i}^{\leftarrow}(Z^i),\,i =1,2, \cdots, d\right),$ 
and  $\nu(\E^{(l-1)}) > 0$ and $\nu(\E^{(l)}) = 0.$ 
\end{enumerate}
In each case, \eqref{non-stan_reg_var_onE_stan} holds,
\eqref{reg_var_on_E0l} holds 
for  $\left( {a^i}^{\leftarrow}(Z^i),\, i =1,2, \cdots, d\right)$ and $t/b^{(l)}(t) \to \infty$.

Recall the definitions of antiranks $\{ r_i^j, i=1,2, \cdots, n,
j=1,2, \cdots, d \},$ $l$-th largest components of $\{ 1/r_i^j, j=1,2,
\cdots, d \}$ denoted $m^{(l)}_i$ for each $i,$ and order statistics
of $\{ m^{(l)}_i, i =1, 2, \cdots, n\},$ denoted $\{ m^{(l)}_{(p)},
p=1, 2, \cdots, n\}.$ Here is a method to detect HRV on $\E^{(l)}$
in the non-standard case.

\begin{prop}\label{rank_conv}
\rm{ Assume that ${\bf{Z}}_1, {\bf{Z}}_2, \ldots {\bf{Z}}_n$ are iid random vectors from a distribution on $[0, \infty)^d$ that satifies both regular variation on $\E$ and HRV on $\E^{(l)}$, so that \eqref{non-stan_reg_var_onE_stan} holds and \eqref{reg_var_on_E0l} holds with ${\bf{Z}}$ replaced by $\left( {a^j}^{\leftarrow}(Z^j), j =1,2, \cdots, d\right)$. We assume that $\nu^{(l)}( \aleph^{(l)}) = 1$. Then, we have on $M_+(\E^{(l)})$,
\begin{equation}\label{est_nu0_mostgeneral_e0l}
\hat \nu^{(l)} := \frac{1}{k}\sum_{i =1}^n \epsilon_{\left( (1/r_i^j)/ m^{(l)}_{(k)}, 1 \le j \le d \right)} \Rightarrow \nu^{(l)} \quad \text{on $M_+(\E^{(l)})$}.
\end{equation}
}
\end{prop}

\begin{proof}
For $l = d,$ the statement is the same as Proposition 2 of
\cite{heffernan:resnick:2005}, except that instead of defining HRV on $\E^{(2)},$ we have assumed HRV on $\E^{(d)}.$  The proof of the case $2 \le l < d$
is  similar to the case for $l=d$ and is omitted.
\end{proof}
\begin{rem}
\rm{In the case, $l=2 < d,$ the only improvement of Proposition
  \ref{rank_conv}, over Proposition 2 of \cite{heffernan:resnick:2005}
  is that here we assume $\nu^{(2)}( \aleph^{(2)}) = 1$ instead of
  assuming $\nu^{(2)}(\{ x \in \E^{(2)} : \wedge_{j=1}^d x^j \ge 1\})
  = 1$. We claim that if HRV on $\E^{(2)}$ is present, the
  assumption $\nu^{(2)}(\aleph^{(2)}) = 1$ could always be achieved by
  a suitable choice of $b^{(2)}(t)$, but if $d > 2$, this may not be
  true for the assumption of $\nu^{(2)}(\{ x \in \E^{(2)} :
  \wedge_{j=1}^d x^j \ge 1\}) = 1$, as claimed in
  \cite{heffernan:resnick:2005}.  See
  Example \ref{iid_pareto} for an illustration.
}
\end{rem}

Proposition \ref{rank_conv} gives us a consistent estimator of
$\nu^{(l)}$, without using the semi-parametric structure of $\nu^{(l)}$
resulting from  \eqref{scaling_e0l} and we now exploit this structure.
In the non-standard case, decomposition of $\nu^{(l)}$ is achieved
as in Proposition \ref{stan_transform}, only the role
of ${\bf{Z}}$ is played by $\left( {a^j}^{\leftarrow}(Z^j), j =1,2,
  \cdots, d\right)$. The limit measure $\nu^{(l)}$ of
\eqref{reg_var_on_E0l} is related to the hidden spectral measure
$S^{(l)}$ through \eqref{nu0ls0l}, which acts as the
definition of the hidden spectral measure $S^{(l)}$ in the non-standard
case.

\begin{prop}\label{spec_trans_mostgeneral}
\rm{The following two statements are equivalent:}
\begin{enumerate}
\rm{ \item The estimator of $\nu^{(l)}$ based on ranks is consistent
  as $k(n) \to \infty$, $k(n)/n \to 0$, and $n \to \infty$; i.e.}
\begin{equation}\label{non_stan_rank_conv}
\hat \nu^{(l)} := \frac{1}{k}\sum_{i =1}^n \epsilon_{\left( (1/r_i^j)/ m^{(l)}_{(k)}, 1 \le j \le d \right)} \Rightarrow \nu^{(l)} \quad \text{on $M_+(\E^{(l)})$}.
\end{equation}

\item \rm{The estimator of $\nu_{\alpha^{(l)}} \times S^{(l)}$ based
    on ranks is consistent as $k(n) \to \infty$, $k(n)/n \to 0$, and
    $n \to \infty$; i.e.}
\begin{equation}\label{non_stan_rank_trans_conv}
 \frac{1}{k}\sum_{i =1}^n \epsilon_{\left( m^{(l)}_i/ m^{(l)}_{(k)}, \hskip 0.1 cm \left( (1/r_i^j)/ m^{(l)}_i, \hskip 0.1cm 1 \le j \le d \right) \right)} \Rightarrow \nu_{\alpha^{(l)}} \times S^{(l)} \quad \text{on $M_+((0, \infty] \times \delta \aleph^{(l)})$}.
\end{equation}
\end{enumerate}
\end{prop}
\begin{proof} See Appendix \ref{proof_prop_spec_most_general}.
\end{proof}

Detection of hidden regular variation on $\E^{(l)}$, for some
$2 \le l \le d,$ requires the prior conclusion that 
 $\left( {a^i}^{\leftarrow}(Z^i),\, i =1,2, \cdots,
    d\right)$
is standard regularly varying on a
bigger sub-cone $\E^{(j)} \supset \E^{(l)}$.
Using the rank transform, we explore for regular variation on $\E$ 
and then move sequentially through the cones $\E
\supset \E^{(2)} \supset \cdots .$  
We also need $\nu^{(j)}$ to satisfy $\nu^{(j)}(\E^{(l-1)}) > 0$
and $\nu^{(j)}(\E^{(l)}) = 0$
which is verified using the hidden spectral measure
$S^{(j)}.$ Finally, we verify regular variation on the cone $\E^{(l)}.$ From
Proposition \ref{rank_conv} and Proposition
\ref{spec_trans_mostgeneral},  HRV on $\E^{(l)}$ implies
\begin{equation}\label{second_component_rank__reg_var}
  \frac{1}{k}\sum_{i =1}^n \epsilon_{ m^{(l)}_i/ m^{(l)}_{(k)}}  \Rightarrow \nu_{\alpha^{(l)}} \quad \text{on $M_+((0, \infty])$}.
\end{equation}
We can use, for example, a  Hill plot to determine whether
\eqref{second_component_rank__reg_var} is true since
consistency of the Hill estimator is only dependent on the consistency
of the tail empirical measure and does not require the tail empirical
measure to be constructed using iid data.
(\cite{resnick:starica:1997a},  \cite[page
80]{resnickbook:2007}).
 This gives us an exploratory method
 for detecting hidden regular variation on $\E^{(l)}$ in the non-standard case. 

To estimate the limit measure $\nu^{(l)}$, it is again
  sufficient to estimate $\alpha^{(l)}$ and the hidden spectral
measure $S^{(l)}$.
Estimate $\alpha^{(l)}$ using, say,
 the Hill estimator
based on the rank-based data $\{ m^{(l)}_i, i = 1, 2, \cdots, n \}$
\cite[Chapter 4]{resnickbook:2007} and using
Proposition \ref{rank_conv} and Proposition
\ref{spec_trans_mostgeneral}, we get in $M_+(\delta \aleph^{(l)})$ that
$
\frac{1}{k} \sum_{i =1}^n \epsilon_{(m^{(l)}_i/ m^{(l)}_{(k)}, \hskip 0.1 cm  ((1/r_i^j)/ m^{(l)}_i, \hskip 0.1cm 1 \le j \le d )) }( [1, \infty]  \times \cdot ) \Rightarrow \nu_{\alpha^{(l)}}([1, \infty])S^{(l)}(\cdot) = S^{(l)}(\cdot)
$ or
\begin{equation}\label{non-stan_spec_est}
\hat S^{(l)}(\cdot) := 
\frac{ \sum_{i =1}^n \epsilon_{\left(m^{(l)}_i/ m^{(l)}_{(k)}, \hskip
      0.1 cm  \left((1/r_i^j)/ m^{(l)}_i, \hskip 0.1cm 1 \le j \le d
      \right) \right) }( [1, \infty]  \times \cdot ) 
}
{
\sum_{i =1}^n \epsilon_{\left(m^{(l)}_i/ m^{(l)}_{(k)} \right)
 }( [1, \infty] )
} \Rightarrow  S^{(l)}(\cdot).
\end{equation}
 This gives a consistent estimator of $S^{(l)}(\cdot)$.

\section{A different representation of the hidden spectral
  measure}\label{diff_para_sec} 

As discussed in the introduction, we map points of $\delta
\aleph^{(l)}$
to the $(d-1)$-dimensional simplex $\Delta_{d-1} = \{ {\bf{x}} \in
[0,1]^{d-1}: \sum_{i=1}^{d-1} x^i \le 1\}$. The probability measure
$\tilde S^{(l)}$ on the transformed points induced by $S^{(l)}$ is
called the transformed (hidden) spectral measure. However, we must
make the standing assumption that 
\begin{equation}\label{eqn:nonoinfinity}
\nu^{(l)}( \{ {\bf{x}} \in \E^{(l)}
: x^{(l)} \ge 1, x^{(1)} = \infty \} ) = 0,\qquad  \text{ for all }2
\le l \le d,
\end{equation}
whenever $\nu^{(l)}$ exists, since otherwise the transformation
is not one-one. Assumption
\eqref{eqn:nonoinfinity}
is not  very strong and most examples
satisfy this assumption. Nonetheless, this assumption is not
always true, as  illustrated by examples in Section
\ref{examples_conc_lim_measure_at_infty}. 
 Recall the conventions that
we replace  $\nu$, $\alpha$, $S$ and $\tilde S$ by $\nu^{(1)}$,
$\alpha^{(1)}$, $S^{(1)}$ and $\tilde S^{(1)}$ respectively.

\subsection{The transformation}\label{diff_par} 

First note that $\nu^{(1)}( \{ {\bf{x}} \in \E^{(1)} : x^{(1)} = \infty \} ) = 0$ due
to the scaling property of $\nu^{(1)}$ in \eqref{scaling_e} and the
compactness of  $\{ {\bf{x}} \in \E^{(1)} : x^{(1)} \ge 1 \}$ in
$\E^{(1)}$. So we may modify \eqref{eqn:nonoinfinity} to include $l=1$.

For each $l$, $1 \le l \le d$, define a transformation
  $T^{(l)}:\delta \aleph^{(l)} \mapsto \Delta_{d-1} =:\{ {\bf{s}} \in [0,1]^{d-1} : \sum_{i=1}^{d-1} s^i \le 1\}$, which is
one-one on an 
appropriate subset of $\delta \aleph^{(l)}$. The
appropriate subset is 
\begin{align}\label{d1l_define}
D_1^{(l)} = \{ {\bf{x}} \in \delta \aleph^{(l)} : x^{(1)} < \infty \}.
\end{align}
On $D_1^{(l)} $,  define $T^{(l)}$ as
\begin{align}\label{define_tl}
T^{(l)}({\bf{x}}) = \frac{( x^2, x^3, \cdots, x^d)}{\sum_{i=1}^d x^i}.
\end{align}
To identify $T^{(l)}(D_1^{(l)})$, first we define a map $\phi^{(l)} : \Delta_{d-1} \to [0,1]$ as
\begin{equation}\label{define_phi}
\phi^{(l)}(s^1, s^2, \cdots, s^{d-1}) = \hbox{the $l$-th largest component of} \hskip 0.2 cm ( 1- \sum_{i=1}^{d-1} s^i, s^1, s^2, \cdots, s^{d-1}).
\end{equation} 
Using this notation, we see that
\begin{equation}\label{d2l_range_of_tl}
D_2^{(l)} := T^{(l)}(D_1^{(l)}) = \{ (s^1, s^2, \cdots, s^{d-1}) \in \Delta_{d-1} : \phi^{(l)}(s^1, s^2, \cdots, s^{d-1}) > 0\} \subset \Delta_{d-1}.
\end{equation} 
To show that $T^{(l)}$ is one-one on $D_1^{(l)} $, we explicitly define the map ${T^{(l)}}^{-1}: D_2^{(l)} \to D_1^{(l)}$ as
\begin{equation}\label{define_inv_tl}
{T^{(l)}}^{-1}(s^1, s^2, \cdots, s^{d-1}) = \frac{( 1- \sum_{i=1}^{d-1} s^i, s^1, s^2, \cdots, s^{d-1})}{\phi^{(l)}(s^1, s^2, \cdots, s^{d-1})} .
\end{equation} 
We extend our definition of $T^{(l)}$ from $D_1^{(l)}$ to the entire
set $\delta \aleph^{(l)}$ by setting $T^{(l)}({\bf{x}}) = {\bf{0}}$
for ${\bf{x}} \in {D_1^{(l)}}^c$. We define a similar extension of
${T^{(l)}}^{-1}$ to the whole simplex $\Delta_{d-1}$ by setting
${T^{(l)}}^{-1}(s^1, s^2, \cdots, s^{d-1}) = {\bf{1}}$ for $(s^1, s^2,
\cdots, s^{d-1}) \in {D_2^{(l)}}^c$. Now define the probability
measure $\tilde S^{(l)} =
S^{(l)} \circ {T^{(l)}}^{-1}$
 on $\Delta_{d-1}$; this is called  the transformed hidden angular measure on $\E^{(l)}$. Note that, 
\begin{align*}
S^{(l)}({D_1^{(l)}}^c) &= \nu^{(l)}(\{ {\bf{x}} \in \E^{(l)} : x^{(l)} \ge 1, \frac{ {\bf{x}} }{x^{(l)}} \in {D_1^{(l)}}^c \}) = \nu^{(l)}(\{ {\bf{x}} \in \E^{(l)} : x^{(l)} \ge 1, x^{(1)} = \infty \}) = 0.
\end{align*}
 Therefore, using  \eqref{d2l_range_of_tl}, we get $\tilde S^{(l)}(D_2^{(l)}) =1$. Since $T^{(l)}$ is one-one on $D_1^{(l)}$ and $ S^{(l)}(D_1^{(l)}) =1$, for any Borel set $ A \subset \delta \aleph^{(l)}$, we can compute $S^{(l)}(A)$ by noting that
\begin{equation}\label{s_and_tildes}
S^{(l)}(A) = S^{(l)}(A \cap D_1^{(l)}) = \tilde S^{(l)}(T^{(l)}(A \cap D_1^{(l)})).
\end{equation}
So, studying the transformed hidden angular measure $\tilde S^{(l)}$ on the nice set $\Delta_{d-1}$ is sufficient to understand the hidden angular measure $S^{(l)}$.

\subsection{Estimation of $\tilde S^{(l)}$ } \label{est_hidden_spec_e0l}

In the standard case, we get from \eqref{stan_spec_est}, 
\begin{equation}\label{stan_est_angular}
\hat S^{(l)}(\cdot)  := \frac{1}{k} \sum_{i =1}^n \epsilon_{\left( Z^{(l)}_i/ Z^{(l)}_{(k)}, \hskip 0.1 cm {\bf{Z}}_i/ Z^{(l)}_i \right)} ([1, \infty] \times \cdot ) \Rightarrow S^{(l)}(\cdot) 
\end{equation}
on $M_+(\delta \aleph^{(l)})$. The function $T^{(l)}$ defined
\eqref{define_tl} 
is continuous on $D_1^{(l)}$ and hence is continuous almost surely
with respect to the probability measure $S^{(l)}$. Therefore, by the
continuous mapping theorem \citep[page 21]{billingsley:1999}, 
\begin{equation}\label{stan_est_angular_on _simplex}
\hat S^{(l)} \circ {T^{(l)}}^{-1}(\cdot) := \frac{1}{k} \sum_{i =1}^n \epsilon_{\left( Z^{(l)}_i/ Z^{(l)}_{(k)}, \hskip 0.1 cm T^{(l)}\left( {\bf{Z}}_i/ Z^{(l)}_i \right) \right)} ([1, \infty] \times \cdot ) \Rightarrow S^{(l)} \circ {T^{(l)}}^{-1}(\cdot) = \tilde S^{(l)} (\cdot) 
\end{equation}
on $M_+(\Delta_{d-1})$. Conversely,  \eqref{stan_est_angular_on _simplex}
implies \eqref{stan_est_angular} by continuity of
${T^{(l)}}^{-1}$ on $D_2^{(l)}$ and the fact $\tilde
S^{(l)}(D_2^{(l)}) = 1$. Thus
\eqref{stan_est_angular} and \eqref{stan_est_angular_on _simplex} are
equivalent. 

In the non-standard case, \eqref{non-stan_spec_est} implies that on $M_+(\delta \aleph^{(l)})$,
\begin{align*}
\hat S^{(l)}(\cdot)  := &\frac{ 1}{k} \sum_{i =1}^n \epsilon_{\left( m^{(l)}_i/ m^{(l)}_{(k)}, \hskip 0.1 cm \left((1/r_i^j)/ m^{(l)}_i, \hskip 0.1cm 1 \le j \le d \right) \right)} ([1, \infty] \times \cdot ) \Rightarrow S^{(l)}(\cdot) .
\end{align*}
 By a similar argument as in the standard case, this is equivalent to the fact that on $M_+(\Delta_{d-1})$,
\begin{equation*}
\hat S^{(l)} \circ {T^{(l)}}^{-1}(\cdot) := \frac{1}{k} \sum_{i =1}^n \epsilon_{\left( m^{(l)}_i/ m^{(l)}_{(k)}, \hskip 0.1 cm T^{(l)}\left( (1/r_i^j)/ m^{(l)}_i , \hskip 0.1cm 1 \le j \le d \right) \right)} ([1, \infty] \times \cdot ) \Rightarrow S^{(l)} \circ {T^{(l)}}^{-1}(\cdot) = \tilde S^{(l)} (\cdot) .
\end{equation*}

\subsection{Supports of transformed (hidden) spectral measure $ \tilde S^{(l)}$} \label{inference_hidden_spec}
The following lemma illustrates that the supports of the transformed (hidden) spectral measures are disjoint. 
 \begin{lem}\label{lemma_2}
\rm{ Recall $D_2^{(l)}$ defined in \eqref{d2l_range_of_tl}. For $1 \le j < l \le d$,
$$\nu^{(j)}(\E^{(l)}) = 0 \hskip 0.1 cm \hbox{iff} \hskip 0.2 cm \tilde S^{(j)}(D_2^{(l)}) = 0.$$
}
\end{lem}
\begin{proof}
By the scaling property \eqref{scaling_e} or \eqref{scaling_e0l}, $\nu^{(j)}(\{ {\bf{x}} \in \E : x^{(j)} = \infty \} ) = 0$, and hence, by the continuous mapping theorem,
 \begin{align*}
 \nu^{(j)}(\E^{(l)}) =  \nu^{(j)}(\E^{(l)} \cap \{ {\bf{x}} \in \E : x^{(j)} < \infty \}) = \nu_{\alpha^{(j)}} \times S^{(j)} (Q^{(j)}(\E^{(l)} \cap \{ {\bf{x}} \in \E : x^{(j)} < \infty \})), 
 \end{align*}
 where $Q^{(j)}({\bf{x}}) = \left( x^{(j)}, \frac{{\bf{x}}}{ x^{(j)}} \right)$. 
Now,
\begin{align*}
\nu_{\alpha^{(j)}} \times S^{(j)} (Q^{(j)}(\E^{(l)} \cap \{ {\bf{x}} \in \E : x^{(j)} < \infty \})) &= \nu_{\alpha^{(j)}} \times S^{(j)}( \{ (r, {\boldsymbol{\theta}}) \in(0, \infty) \times \delta \aleph^{(j)}: \theta^{(l)} > 0 \} ) \\
&= \lim_{\lambda \to 0} \lambda^{-\alpha^{(j)}} S^{(j)}( \{ {\boldsymbol{\theta}} \in \delta \aleph^{(j)}: \theta^{(l)} > 0 \} )
\end{align*}
Hence, $\nu^{(j)}(\E^{(l)}) = 0$ iff $S^{(j)}( \{ {\boldsymbol{\theta}} \in \delta \aleph^{(j)}: \theta^{(l)} > 0 \} ) = 0$. Since $S^{(j)}(D_1^{(j)}) = 1$, where $D_1^{(j)}$ is as given in \eqref{d1l_define}, we get
 \begin{align*}
S^{(j)}( \{ {\boldsymbol{\theta}} \in \delta \aleph^{(j)}: \theta^{(l)} > 0 \} ) &= S^{(j)}( \{ {\boldsymbol{\theta}} \in \delta \aleph^{(j)}: \theta^{(l)} > 0 \} \cap D_1^{(j)})\\
&= \tilde S^{(j)}( T^{(j)}(\{ {\boldsymbol{\theta}} \in \delta \aleph^{(j)}: \theta^{(l)} > 0 \} \cap D_1^{(j)} ))\\
&= \tilde S^{(j)}(\{ (s^1, s^2, \cdots, s^{d-1}) \in \Delta_{d-1} : \phi^{(l)}(s^1, s^2, \cdots, s^{d-1}) > 0 \}) \\
&= \tilde S^{(j)}(D_2^{(l)}).
\end{align*}
Hence, the result follows.
\end{proof}

\begin{rem}\label{rem_lemma_2}
 \rm{The fact that $\nu^{(j)}(\E^{(l)}) = 0$ iff $S^{(j)}(\{
   {\boldsymbol{\theta}} \in \delta \aleph^{(j)} : \theta^{(l)} > 0 \}
   ) = 0$, follows from the proof of Lemma \ref{lemma_2}. Notice, this
   result does not require the assumption \eqref{eqn:nonoinfinity}. } 
 \end{rem}

 If $\nu^{(j)}(\E^{(l)}) = 0$ and HRV on $\E^{(l)}$ exists, then
 the support of $\tilde S^{(j)}$ is contained in ${D_2^{(l)}}^c$ and
 the support of $\tilde S^{(l)}$ is contained in $D_2^{(l)}$, which
 are disjoint. So, if one  seeks (hidden) regular variation on the
nested cones $\E = \E^{(1)} \supset \E^{(2)} \supset \cdots \supset
 \E^{(d)}$, if HRV is present, the 
 transformed spectral measure and the transformed hidden spectral
 measures on $\Delta_{d-1}$ will have disjoint supports. 
 
 For a visual illustration,  fix $d = 3$ and suppose
 $\tilde S^{(1)}$ is concentrated on the corner points of the triangle
 $\Delta_2$. 
By Lemma \ref{lemma_2},  $\nu^{(1)}(\E^{(2)}) = 0$ and we search for
HRV on $\E^{(2)}$. Assume that it is indeed present and so consider
$\tilde S^{(2)}$. As we have already noticed, the support of 
$\tilde S^{(2)}$ is contained in $D_2^{(2)}$ and hence does not put
any mass on the corner points of the triangle $\Delta_2$. Therefore,
$\tilde S^{(2)}$ and $\tilde S^{(1)}$ have disjoint supports. Two
cases might arise from this situation. In the first case,
$\tilde S^{(2)}$ puts positive mass in the interior of the triangle
$\Delta_2$. Applying Lemma \ref{lemma_2}, we infer that
$\nu^{(2)}(\E^{(3)}) > 0$ which rules out the possibility of 
 HRV on $\E^{(3)}$. Hence, we do not consider $\tilde
S^{(3)}$. In the second case, $\tilde S^{(2)}$ is concentrated on the
axes of the triangle $\Delta_2$ and by Lemma \ref{lemma_2},
$\nu^{(2)}(\E^{(3)}) = 0$. Hence, as usual, we search for HRV on
$\E^{(3)}$ and let us assume that it is present. Then, we 
consider $\tilde S^{(3)}$. As noted, the support of $\tilde
S^{(3)}$ is contained in $D_2^{(3)}$ and hence it only puts mass in
the interior of the triangle $\Delta_2$. Hence, in this case, all
three of $\tilde S^{(1)}$, $\tilde S^{(2)}$ and $\tilde S^{(3)}$
have disjoint supports.  
  
 Now, consider another case, where $\tilde S^{(1)}$ is not
 concentrated on the corner points of the triangle $\Delta_2$, but is
 concentrated on its axes. Using Lemma \ref{lemma_2}, 
$\nu(\E^{(2)}) > 0$, but $\nu^{(1)}(\E^{(3)}) = 0$. So, we
 should not search for HRV on $\E^{(2)}$ and hence should not consider
 $\tilde S^{(2)}$. However, we  consider presence of HRV on
 $\E^{(3)}$ and hence consider $\tilde S^{(3)}$. But, the support of
 $\tilde S^{(3)}$ is contained in the interior of the triangle
 $\Delta_2$ and hence $\tilde S^{(3)}$ does not put any mass on the
 axes. So, in this case also, we would consider only $\tilde S^{(3)}$
 and $\tilde S^{(1)}$, which have disjoint supports. 
 
 In the final case, suppose $\tilde S^{(1)}$ puts mass in the interior
 of the triangle $\Delta_{2}$.  Lemma \ref{lemma_2}
 implies $\nu^{(1)}(\E^{(3)}) > 0$
and we should not seek HRV on any of the sub-cones $\E^{(2)}$ or $\E^{(3)}$. 
 
In all these illustrative cases, 
 the transformed spectral measure and the transformed hidden spectral measures have disjoint supports.

\subsection{Lines through ${\boldsymbol{\infty}}$ } \label{examples_conc_lim_measure_at_infty}
Section \ref{diff_para_sec} made the standing assumption \eqref{eqn:nonoinfinity}, which is not always true. In
Example \ref{mass_at_infty_ex1}, the measure $\nu^{(2)}$ 
concentrates on the lines through $\infty;$ i.e., on the set $\{
{\bf{x}} \in \E^{(2)} : x^{(1)} = \infty \}.$ 
Examples \ref{mass_at_infty_ex2} and \ref{mass_at_infty_ex3} show that
for $2 \le j < l \le d,$ $\nu^{(l)}( \{ {\bf{x}} \in \E^{(l)} :
x^{(l)} \ge 1, x^{(1)} = \infty \} ) = 0$ does not imply $\nu^{(j)}(
\{ {\bf{x}} \in \E^{(j)} : x^{(j)} \ge 1, x^{(1)} = \infty \} ) = 0$
and vice versa.

 \begin{ex}\label{mass_at_infty_ex1}
\rm{
Let $X$ and $Y$ be two iid Pareto($1$) random variables. Let $B$ be another random variable independent of $(X,Y)$ such that $P[ B = 0] = P[ B = 1] = \frac{1}{2}$. Define
$$ {\bf{Z}} = (Z^1, Z^2) = B(X, X^2) + (1-B)(Y^2, Y),$$
so that
\begin{equation*}
 tP\left[ {{\bf{Z}}}/{t^2} \in \cdot \,\right] \stackrel{v}{\rightarrow} \nu(\cdot) \quad \text{in $M_+(\E)$},
\end{equation*}
where for $w, v > 0$, $\nu((w, \infty] \times [0, \infty]) =
\frac{1}{2} w^{-1/2}$, $\nu([0, \infty] \times (v, \infty] ) =
\frac{1}{2} v^{-1/2}$ and $\nu(\E^{(2)}) = 0$. 
For $w, v > 0$,
\begin{align*}
\lim_{t \to \infty} tP&\left[ \frac{{\bf{Z}}}{t} \in  (w, \infty] \times
  (v, \infty] \right]= \lim_{t \to \infty} \frac{t}{2}P\left[  
X > tw, X^2 > tv \right] + \lim_{t \to \infty} \frac{t}{2}P\left[ Y^2 > tw, Y> tv \right] \\
&= \lim_{t \to \infty} \frac{t}{2}P\left[ X > tw \right] + \lim_{t \to
  \infty} \frac{t}{2}P\left[ Y> tv \right] 
= \frac{1}{2}\left(\frac{1}{w} + \frac{1}{v} \right).
\end{align*}
So HRV exists on the cone $\E^{(2)}$ with limit measure $\nu^{(2)}$ such that 
\begin{equation*}
\nu^{(2)}((w, \infty] \times (v, \infty] )= \frac{1}{2}\left(\frac{1}{w} + \frac{1}{v} \right).
\end{equation*}
Hence, letting $v \to \infty$, we get $\nu^{(2)}((w, \infty] \times \{ \infty \} )= \frac{1}{2w}$
and similarly, $\nu^{(2)}(\{ \infty \} \times (v, \infty] )= \frac{1}{2v}.$ So, we conclude that in this case, $\nu^{(2)}( \{ {\bf{x}} \in \E^{(l)} : x^{(2)} \ge 1, x^{(1)} = \infty \} ) = 1$.
}
\end{ex}

\begin{ex}\label{mass_at_infty_ex2}
\rm{Let $X_1, X_2, \cdots, X_5$ be five iid Pareto($1$) random variables. Let $(B_1, B_2, B_3)$ be another set of random variables independent of $(X_1, X_2, \cdots, X_5)$ such that $P[ B_i = 1] = 1 - P[ B_i = 0] = \frac{1}{3}$ and $\sum_{i=1}^3 B_i =1$. Now, define ${\bf{Z}}$ as
$$ {\bf{Z}} = (Z^1, Z^2, Z^3) = B_1(X_1, X_1^2, 0) + B_2(X_2^2, X_2, 0) + B_3(X_3^2, X_4^2, X_5^2).$$
It follows that 
\begin{equation*}
tP\left[ \frac{{\bf{Z}}}{25t^2/9} \in \cdot \right] \stackrel{v}{\rightarrow} \nu(\cdot) \quad \text{in $M_+(\E)$},
\end{equation*}
where for $w, v, x > 0$, $\nu((w, \infty] \times [0, \infty] \times [0, \infty] ) = \frac{2}{5} w^{-1/2}$, $\nu([0, \infty] \times (v, \infty] \times [0, \infty] ) = \frac{2}{5} v^{-1/2}$, $\nu( [0, \infty] \times [0, \infty] \times (x, \infty] ) = \frac{1}{5} x^{-1/2}$ and $\nu(\E^{(2)}) = 0$. Now, we look for HRV on $\E^{(2)}$. Notice that 
\begin{equation*}
tP\left[ \frac{{\bf{Z}}}{5t/3} \in \cdot \right] \stackrel{v}{\rightarrow} \nu^{(2)}(\cdot) \quad \text{in $M_+(\E^{(2)})$},
\end{equation*}
where for $w, v, x > 0$, $\nu^{(2)}((w, \infty] \times (v, \infty] \times [0, \infty] ) = \frac{1}{5} \left(w^{-1} + v^{-1} + {(wv)}^{-1/2} \right)$, $\nu^{(2)}([0, \infty] \times (v, \infty] \times (x, \infty] ) = \frac{1}{5} {(vx)}^{-1/2}$, $\nu^{(2)}( (w, \infty] \times [0, \infty] \times (x, \infty] ) = \frac{1}{5} {(xw)}^{-1/2}$ and $\nu^{(2)}(\E^{(3)}) = 0$. Hence, letting $v \to \infty$, we get
\begin{equation*}
\nu^{(2)}((w, \infty] \times \{ \infty \} \times [0, \infty] )= \frac{1}{5w},
\end{equation*}
and so $\nu^{(2)}( \{ {\bf{x}} \in \E^{(2)} : x^{(2)} \ge 1, x^{(1)} =
\infty \} ) > 0$. 
We now seek HRV on the cone $\E^{(3)}$. For $w, v, x > 0$,
\begin{align*}
\lim_{t \to \infty} tP&\left[ \frac{{\bf{Z}}}{{(t/3)}^{2/3}} \in (w, \infty] \times (v, \infty] \times (x, \infty] \right] \\
&= \lim_{t \to \infty} \frac{t}{3}P\left[ X_3^2 > {(t/3)}^{2/3}w, X_4^2> {(t/3)}^{2/3}v, X_5^2 > {(t/3)}^{2/3}x \right] = {(wvx)}^{-1/2}.
\end{align*}
So, HRV exists on the cone $\E^{(3)}$ with limit measure $\nu^{(3)}$ such that for $w, v, x > 0$,
\begin{equation*}
\nu^{(3)}((w, \infty] \times (v, \infty] \times (x, \infty] )= {(wvx)}^{-1/2}.
\end{equation*}
Hence, for this example, $\nu^{(3)}( \{ {\bf{x}} \in \E^{(3)} : x^{(3)} \ge 1, x^{(1)} = \infty \} ) = 0$ and thus for $2 \le j < l \le d,$ $\nu^{(l)}( \{ {\bf{x}} \in \E^{(l)} : x^{(l)} \ge 1, x^{(1)} = \infty \} ) = 0$ does not imply $\nu^{(j)}( \{ {\bf{x}} \in \E^{(j)} : x^{(j)} \ge 1, x^{(1)} = \infty \} ) = 0$.
}
\end{ex}

 \begin{ex}\label{mass_at_infty_ex3}
\rm{Let $X_1, X_2, \cdots, X_5$ be five iid Pareto($1$) random
  variables. Let $(B_1, B_2, B_3)$ be another set of random variables
  independent of $(X_1, X_2, \cdots, X_5)$ such that $P[ B_i = 1] = 1
  - P[ B_i = 0] = \frac{1}{3}$ and $\sum_{i=1}^3 B_i =1$. Now, define
  ${\bf{Z}}$ as 
$$ {\bf{Z}} = (Z^1, Z^2, Z^3) = B_1(X_1, X_1^3, X_1^{5/4}) + B_2(X_2^3, X_2, X_2^{5/4}) + B_3(X_3^3, X_4^3, X_5^3).$$
It follows that
\begin{equation*}
tP\left[ \frac{{\bf{Z}}}{125t^3/27} \in \cdot \right] \stackrel{v}{\rightarrow} \nu(\cdot) \quad \text{in $M_+(\E)$},
\end{equation*}
 where for all $w, v, x > 0$, $\nu((w, \infty] \times [0, \infty]
 \times [0, \infty] ) = \frac{2}{5} w^{-1/3}$, $\nu([0, \infty] \times
 (v, \infty] \times [0, \infty] ) = \frac{2}{5} v^{-1/3}$, $\nu( [0,
 \infty] \times [0, \infty] \times (x, \infty] ) = \frac{1}{5}
 x^{-1/3}$ and $\nu(\E^{(2)}) = 0$. Now, when we seek HRV on
 $\E^{(2)}$, we get
 \begin{equation*}
tP\left[ \frac{{\bf{Z}}}{ t^{3/2}} \in \cdot \right] \stackrel{v}{\rightarrow} \nu^{(2)}(\cdot) \quad \text{in $M_+(\E^{(2)})$},
\end{equation*}
where for $w, v, x > 0$, $\nu^{(2)}((w, \infty] \times (v, \infty]
\times [0, \infty] ) = \frac{1}{3} {(wv)}^{-1/3}$, $\nu^{(2)}([0,
\infty] \times (v, \infty] \times (x, \infty] ) = \frac{1}{3}
{(vx)}^{-1/3}$, $\nu^{(2)}( (w, \infty] \times [0, \infty] \times (x,
\infty] ) = \frac{1}{3} {(xw)}^{-1/3}$ and $\nu^{(2)}(\E^{(3)}) =
0$. Notice, $\nu^{(2)}( \{ {\bf{x}} \in \E^{(2)} : x^{(2)} \ge 1,
x^{(1)} = \infty \} ) = 0$. Now, we look for HRV on the cone
$\E^{(3)}$. For $w, v, x > 0$, 
\begin{align*}
\lim_{t \to \infty} tP&\left[ \frac{{\bf{Z}}}{t} \in (w, \infty] \times (v, \infty] \times (x, \infty] \right] \\
&=  \lim_{t \to \infty} \frac{t}{3}P\left[ X_1 > tw, X_1^3> tv, X_1^{5/4} > tx \right] \\
& \qquad 
+  \lim_{t \to \infty} \frac{t}{3}P\left[ X_2^3 > tw, X_2> tv, X_2^{5/4} > tx \right] \\
& \qquad 
+ \lim_{t \to \infty} \frac{t}{3}P\left[ X_3^3 > tw, X_4^3 > tv, X_5^3 > tx \right] \\
&=  \lim_{t \to \infty} \frac{t}{3}P\left[ X_1 > tw \right] + \lim_{t \to \infty} \frac{t}{3}P\left[ X_2 > tv \right] \\
& \qquad + \lim_{t \to \infty} \frac{t}{3}P\left[ X_3 > {(tw)}^{1/3}, X_4 > {(tv)}^{1/3}, X_5 > {(tx)}^{1/3} \right] \\
&= \frac{1}{3}\left( w^{-1} + v^{-1} + {(wvx)}^{-1/3} \right).
\end{align*}
So, HRV exists on the cone $\E^{(3)}$ with limit measure $\nu^{(3)}$ such that 
\begin{equation*}
\nu^{(3)}((w, \infty] \times (v, \infty] \times (x, \infty] )= \frac{1}{3}\left( w^{-1} + v^{-1} + {(wvx)}^{-1/3} \right).
\end{equation*}
Following Example \ref{mass_at_infty_ex1}, 
 $\nu^{(3)}( \{ {\bf{x}} \in \E^{(3)} : x^{(3)} \ge 1, x^{(1)} =
 \infty \} ) = 2/3$ so that
for $2 \le j < l \le d,$ $\nu^{(j)}( \{ {\bf{x}} \in \E^{(j)} : x^{(j)} \ge 1, x^{(1)} = \infty \} ) = 0$ does not imply $\nu^{(l)}( \{ {\bf{x}} \in \E^{(l)} : x^{(l)} \ge 1, x^{(1)} = \infty \} ) = 0$.
}
\end{ex}

\section{Deciding finiteness of $\nu^{(l)}(\{ {\bf{x}} \in \E^{(l)} : ||{\bf{x}}|| > 1 \})$ }\label{decide_finite_sec}

For characterizations of HRV \citep{maulik:resnick:2005}, 
it is useful to characterize
when 
$\nu^{(l)}(\{ {\bf{x}} \in \E^{(l)} : ||{\bf{x}}|| > 1 \})$ is finite
and when it is not, where $||{\bf{x}}||$ is any norm of ${\bf{x}}$. 
Such characterizations are also useful for estimating
risk set probabilities. For example,
the limit measure $\nu^{(l)}$ puts a finite mass on a risk set of the
form $\{ {\bf{x}} \in \E^{(l)} : a_1x^1 + a_2x^2 + \cdots + a_dx^d > y
\}, a_i > 0, i= 1, 2, \cdots, d, y > 0$, iff $\nu^{(l)}(\{ {\bf{x}}
\in \E^{(l)} : ||{\bf{x}}|| > 1 \})$ is finite. The HRV theory is not
useful for estimation of risk set probability if the limit measure puts infinite mass on that
risk region.

The following section resolves this issue using a moment condition. 
Subsequently we show that for $2 \le j < l \le
d$, neither $\nu^{(l)}(\{ {\bf{x}} \in \E^{(l)} : ||{\bf{x}}|| > 1
\})$ being finite implies  $\nu^{(j)}(\{ {\bf{x}} \in \E^{(j)} :
||{\bf{x}}|| > 1 \})$ is finite, nor the reverse is true.

\subsection{A moment condition} The following theorem gives a
necessary and sufficient condition for the finiteness of $\nu^{(l)}(\{
{\bf{x}} \in \E^{(l)} : ||{\bf{x}}|| > 1 \})$. 
For $d=2$, the condition of Theorem \ref{finite_infinite_thm1} is given in
  Proposition 5.1 of \cite{maulik:resnick:2005}.

\begin{thm}\label{finite_infinite_thm1}
\rm{For each $l$, $2 \le l \le d$, the limit measure $\nu^{(l)}$ puts finite mass on the set $\{ {\bf{x}} \in \E^{(l)} : ||{\bf{x}}|| > 1 \}$, i.e. $\nu^{(l)}(\{ {\bf{x}} \in \E^{(l)} : ||{\bf{x}}|| > 1 \})$ is finite iff
\begin{equation}\label{norm_finite_condn}
\int_{\delta \aleph^{(l)}} {||{\boldsymbol{\theta}}||}^{\alpha^{(l)}} S^{(l)}(d{\boldsymbol{\theta}}) < \infty.
\end{equation}
}
\end{thm}

\begin{proof} We have,
\begin{align*} 
\nu^{(l)} &(\{ {\bf{x}} \in \E^{(l)} : ||{\bf{x}}|| > 1 \}) = \nu^{(l)}(\{ {\bf{x}} \in \E^{(l)} : x^{(l)}||\frac{{\bf{x}}}{x^{(l)}}|| > 1 \})\\
&= \nu_{\alpha^{(l)}} \times S^{(l)}(\{ (r, {\boldsymbol{\theta}} ) \in (0, \infty] \times \delta \aleph^{(l)} : r ||{\boldsymbol{\theta}} || > 1 \})\\
&= \int_{\delta \aleph^{(l)}}   \nu_{\alpha^{(l)}}(\{ r \in (0, \infty]  : r  > 1/||{\boldsymbol{\theta}} || \})S^{(l)}(d{\boldsymbol{\theta}}) 
= \int_{\delta \aleph^{(l)}} {||{\boldsymbol{\theta}}||}^{\alpha^{(l)}} S^{(l)}(d{\boldsymbol{\theta}}).
\end{align*}
Hence, the result follows.
\end{proof}

The following  corollaries  translate the condition of Theorem
\ref{finite_infinite_thm1} 
to the transformed hidden angular measure $\tilde S^{(l)}$.

\begin{cor}
\rm{ If $\nu^{(l)}( \{ {\bf{x}} \in \E^{(l)} : x^{(l)} \ge 1, x^{(1)} = \infty \} ) > 0$, then $\nu^{(l)}(\{ {\bf{x}} \in \E^{(l)} : ||{\bf{x}}|| > 1 \})$ is infinite.
}
\end{cor}
\begin{proof} Observe, if we denote the largest component of ${\boldsymbol{\theta}}$ as ${\theta}^{(1)}$, we get
\begin{equation*}
S^{(l)}(\{ {\boldsymbol{\theta}} \in \delta \aleph^{(l)} : {\theta}^{(1)} = \infty \}) = \nu^{(l)}( \{ {\bf{x}} \in \E^{(l)} : x^{(l)} \ge 1, x^{(1)} = \infty \} ) > 0. 
\end{equation*}
Hence, the result follows from Theorem \ref{finite_infinite_thm1}.
\end{proof}

\begin{cor}
\rm{Suppose, $\nu^{(l)}( \{ {\bf{x}} \in \E^{(l)} : x^{(l)} \ge 1, x^{(1)} = \infty \} ) = 0$. Then, $\nu^{(l)}(\{ {\bf{x}} \in \E^{(l)} : ||{\bf{x}}|| > 1 \})$ is finite iff
\begin{equation}\label{norm_finite_condn_on_simplex}
\int_{D_2^{(l)}} {\left(\frac{ ||(1 - \sum_{i=1}^{d-1} s^i, s^1, s^2, \cdots, s^{d-1})||}{\phi^{(l)}(s^1, s^2, \cdots, s^{d-1})} \right)}^{\alpha^{(l)}} \tilde S^{(l)}(d{\bf{s}}) < \infty,
\end{equation}
where $\phi^{(l)}$ and $D^{(l)}_2$ are defined in \eqref{define_phi} and \eqref{d2l_range_of_tl} respectively.
}
\end{cor}
\begin{proof}
The condition $S^{(l)}({D_1^{(l)}}^c) = \nu^{(l)}( \{ {\bf{x}} \in \E^{(l)} : x^{(l)} \ge 1, x^{(1)} = \infty \} ) = 0,$ where $D_1^{(l)}$ is defined in \eqref{d1l_define}, allows us to apply the change of variable formula to \eqref{norm_finite_condn} using the almost surely one-one transformation $T^{(l)}$ as in \eqref{define_tl}. Now, the result follows from Theorem \ref{finite_infinite_thm1}.
\end{proof}

Choosing the $L_1$-norm in \eqref{norm_finite_condn_on_simplex}, we get the simple condition: $\nu^{(l)}(\{ {\bf{x}} \in \E^{(l)} : ||{\bf{x}}|| > 1 \})<\infty$  iff
\begin{equation*}
\int_{D_2^{(l)}} {\left(\phi^{(l)}(s^1, s^2, \cdots, s^{d-1}) \right)}^{-\alpha^{(l)}} \tilde S^{(l)}(d{\bf{s}}) < \infty.
\end{equation*}

\subsection{A particular construction} We  defined HRV on a series
of sub-cones $\E \supset \E^{(2)} \supset \E^{(3)} \supset \cdots
\supset \E^{(d)}$, and discussed the finiteness condition in Theorem
\ref{finite_infinite_thm1} for each of the limit measures $\nu^{(l)},
\hskip 0.1 cm 2 \le l \le d$. A natural question is if for some $2
\le j < l \le d$, HRV exists on both the 
cones $\E^{(j)}$ and $\E^{(l)}$, does finiteness of $\nu^{(j)}(\{
{\bf{x}} \in \E^{(j)} : ||{\bf{x}}|| > 1 \})$ imply finiteness of
$\nu^{(l)}(\{ {\bf{x}} \in \E^{(l)} : ||{\bf{x}}|| > 1 \})$ or vice
versa? We construct an example to show that there are no such
implications.

\begin{ex}
\rm{
Suppose, $X_i, \hskip 0.1 cm i = 1, 2, \cdots, d$ are iid Pareto(1). Also, assume $R_i, \hskip 0.1 cm i = 2, 3, \cdots, d$ are mutually independent random variables with $R_i$ having distribution Pareto$\left( \frac{i(i+1)}{2i +1} \right)$. Now, for each $2 \le l \le d$, define a set of mutually independent random variables ${\bf{s}}_i, \hskip 0.1 cm i =2, 3, \cdots, d$, such that ${\bf{s}}_i$ has a distribution $\bar S^{(i)}$ on $\{ {\bf{x}} \in D_2^{(i)} : x^{i} = x^{i+1} =  \cdots = x^{d-1}= 0 \}$, where $D_2^{(i)}$ is defined in 
\eqref{d2l_range_of_tl}. Also, assume that $(X_i, i =1, 2, \cdots, d)$, $(R_i, i = 2, 3, \cdots, d)$ and $({\bf{s}}_i, \hskip 0.1 cm i =2, 3, \cdots, d)$ are independent of each other. Note that even though we have restricted the supports of the probability measures $\bar S^{(i)}$, we still have the flexibility to choose them in a way so that \eqref{norm_finite_condn_on_simplex} is satisfied or not, depending on whether we want to make $\nu^{(i)}(\{ {\bf{x}} \in \E^{(i)} : ||{\bf{x}}|| > 1 \})$ finite or infinite.

Now, let $(B_1, B_2, \cdots, B_d)$ be another set of random variables independent of all the previous random variables such that $P[ B_i = 1] = 1 - P[ B_i = 0] = \frac{1}{d}$ and $\sum_{i=1}^d B_i =1$. Recall  the definition of the transformation ${T^{(l)}}^{-1}$ from \eqref{define_inv_tl}, which maps points from $D_2^{(l)}$ to $\delta \aleph^{(l)} = \{ {\bf{x}} \in \E^{(l)} : x^{(l)} = 1 \}$. Note that, the range of ${T^{(l)}}^{-1}$ is $D_1^{(l)},$ where $D_1^{(l)}$ is defined in \eqref{d1l_define}. Now, define the random vector ${\bf{Z}}$ as
\begin{align*}
{\bf{Z}} &= (Z^1, Z^2, \cdots, Z^d) \\
&= B_1(X_1, X_2, \cdots, X_d) + B_2R_2{T^{(2)}}^{-1}({\bf{s}}_2) + B_3R_3{T^{(3)}}^{-1}({\bf{s}}_3) + \cdots + B_dR_d{T^{(d)}}^{-1}({\bf{s}}_d).
\end{align*}
Since the range of ${T^{(l)}}^{-1}$ is $D_1^{(l)}$, all the components
of ${T^{(l)}}^{-1}({\bf{s}}_l)$ are finite, $2 \le l \le d$, and hence
all the components of ${\bf{Z}}$ are $[0, \infty)$-valued. Also, 
\begin{equation*}
 tP\left[ {{\bf{Z}}}/{t} \in \cdot \,\right] \stackrel{v}{\rightarrow}  \nu (\cdot) \quad \text{in $M_+(\E)$},
\end{equation*}
where $\nu( [0, \infty] \times \cdots \times [0, \infty] \times (u,
\infty] \times [0, \infty] \times \cdots \times [0, \infty]) =
{(d\cdot u)}^{-1}$, where $(u, \infty]$ is in the $i$-th position and
this holds for all $1 \le i \le d$. Also, $\nu(\E^{(2)}) = 0.$ Notice,
for each $2 \le l \le d$, the parameter of the distribution of $R_l$
is chosen in such a way that HRV of $(X_1, X_2,
\cdots, X_d)$ on $\E^{(l)}$ or regular variation of
$R_p{T^{(p)}}^{-1}({\bf{s}}_p)$ on $\E^{(l)}$, $l < p \le d$, does not
affect the HRV of ${\bf{Z}}$ on $\E^{(l)}$. Also,
by choosing the support of $\bar S^{(p)}$, $2 \le p \le d$, to be
concentrated on $\{ {\bf{x}} \in D_2^{(p)} : x^{p} = x^{p+1} =
\cdots = x^{d-1} = 0 \}$ we have ensured that
$R_p{T^{(p)}}^{-1}({\bf{s}}_p)$, $ 2 \le p < l$, would not have  any
HRV on the cone $\E^{(l)}$. So, the only part of
${\bf{Z}}$ contributing in HRV on $\E^{(l)}$ is
$R_l{T^{(l)}}^{-1}({\bf{s}}_l)$, and therefore, for $2 \le l \le d$
and $x > 0$, 
\begin{align*}
\lim_{t \to \infty} tP&\left[ \frac{Z^{(l)}}{{(t/d)}^{(2l +1)/l(l+1)}} > x, \frac{{\bf{Z}}}{Z^{(l)}} \in \cdot \right] = \lim_{t \to \infty}\frac{t}{d}P\left[ \frac{R_l }{{(t/d)}^{(2l +1)/l(l+1)}} > x, {T^{(l)}}^{-1}({\bf{s}}_l) \in \cdot \right]\\
&=  \lim_{t \to \infty}\frac{t}{d}P\left[ \frac{R_l }{{(t/d)}^{(2l +1)/l(l+1)}} > x\right] P\left[ {T^{(l)}}^{-1}({\bf{s}}_l) \in \cdot \right]\\
&=  \lim_{t \to \infty}\frac{t}{d}{\left({(t/d)}^{(2l +1)/l(l+1)}x
  \right)}^{- l(l+1)/(2l+1)} P\left[ {T^{(l)}}^{-1}({\bf{s}}_l) \in
  \cdot \right]
\\
&= x^{- l(l+1)/(2l+1)} P\left[ {T^{(l)}}^{-1}({\bf{s}}_l) \in \cdot \right].
\end{align*}
Hence, following Proposition \ref{stan_transform}, for $2 \le l \le d$, ${\bf{Z}}$ has regular variation on the cone $\E^{(l)}$ with scaling function $b^{(l)}(t) = {(t/d)}^{(2l +1)/l(l+1)}$, $\alpha^{(l)} = l(l+1)/(2l+1)$ and
hidden spectral measure $S^{(l)}(\cdot) =  P\left[ {T^{(l)}}^{-1}({\bf{s}}_l) \in \cdot \right]$. Also, notice
$1/ \alpha^{(l)} = \left( \frac{1}{l} + \frac{1}{l+1} \right)$ is a decreasing function in $l$, which indeed confirms that for $2 \le j < l \le d$, $b^{(j)}(t)/ b^{(l)}(t) \to \infty$, which is a required condition for HRV on $\E^{(l)}$. So, ${\bf{Z}}$ has HRV on the each of the cones $\E^{(l)}$ with the limit measure $\nu^{(l)},\hskip 0.1 cm 2 \le l \le d$. Now, we look for the transformed hidden spectral measure $\tilde S^{(l)}$ for the limit measure $\nu^{(l)}$ and show that it indeed coincides with $\bar S^{(l)}$.

Since the hidden spectral measure $S^{(l)}$ has been defined through the function ${T^{(l)}}^{-1}$ which has range $D_1^{(l)}$, we have $S^{(l)}(D_1^{(l)}) =1$, where $D_1^{(l)},$ where $D_1^{(l)}$ is defined in \eqref{d1l_define}. So, we get the transformed hidden spectral measure $\tilde S^{(l)}(\cdot)$ as $\tilde S^{(l)}(\cdot) = P\left[ {\bf{s}}_l \in \cdot \right]$. So, this hidden transformed spectral measure $\tilde S^{(l)}$ matches with our earlier $\bar S^{(l)}$.
Following the comments made before about $\bar S^{(l)}$, we have the flexibility to choose $\tilde S^{(l)}$ in such a way that \eqref{norm_finite_condn_on_simplex} is satisfied or not, and this could be done independently for each $2 \le l \le d$. So, this example shows that we could construct a random variable which has regular variation on each of the cones $\E^{(l)}$ with limit measure $\nu^{(l)}, \hskip 0.1 cm 2 \le l \le d$,  and for each $ 2 \le l \le d$, we could independently choose to make $\nu^{(l)}(\{ {\bf{x}} \in \E^{(l)} : ||{\bf{x}}|| > 1 \})$ finite or infinite. Therefore, for $2 \le j < l \le d$, neither $\nu^{(j)}(\{ {\bf{x}} \in \E^{(j)} : ||{\bf{x}}|| > 1 \})$ is finite implies $\nu^{(l)}(\{ {\bf{x}} \in \E^{(l)} : ||{\bf{x}}|| > 1 \})$ is finite , nor the reverse is true.
}
\end{ex}

\section{Computation of probabilities of risk sets}\label{risksetcomputation}
In this section, we consider two risk regions and illustrate how HRV helps obtain
more accurate estimates of probabilities of risk sets.

\subsection{At least one risk is large.}
One scenario has ${\bf{Z}} =
(Z^1, Z^2, \cdots, Z^d)$ representing risks such as pollutant
concentrations at $d$ sites \cite{heffernan:resnick:2005}. A critical risk level, such as
  pollutant concentration $t^i$ ($i = 1, 2, \cdots,
d$) at
the $i$-th site, is set by a government agency.
 Exceeding $t^i$ for some $i$
results in a fine and the event
{\it non-compliance\/} is represented by $ \cup_{i=1}^d [ Z^i > t^i]$.
The probability of non-compliance is,
\begin{align*}
P[ \hbox{non-compliance}] &= P[ \cup_{i=1}^d \{ Z^i > t^i \} ] = \sum_iP[Z^i > t^i] - \sum_{ 1 \le i_1 < i_2 \le d} P[Z^{i_1} > t^{i_1}, Z^{i_2} > t^{i_2}]\\
& \quad + \cdots + {(-1)}^{(j-1)}\sum_{1 \le i_1 < i_2 < \cdots i_j \le d} P[Z^{i_1} > t^{i_1}, Z^{i_2} > t^{i_2}, \cdots, Z^{i_j} > t^{i_j} ]\\
& \qquad + \cdots + {(-1)}^{(d-1)}P[Z^{1} > t^{1}, Z^{2} > t^{2}, \cdots, Z^{d} > t^{d} ].
\end{align*}
Suppose, ${\bf{Z}}, {\bf{Z}}_1, {\bf{Z}}_2, \cdots, {\bf{Z}}_n$ are
iid random vectors whose common distribution, for simplicity, is
  assumed standard regularly  varying on $\E = \E^{(1)}$ with scaling
  function $b(t) = b^{(1)}(t)$ as in \eqref{stan_reg_var}.  Assume HRV
  holds on each of the cones $\E^{(l)}$ with scaling 
  function $b^{(l)}(t)$ as in \eqref{reg_var_on_E0l}, $2 \le l \le
  d$. Since asymptotic independence is present, relying only on
    regular variation on $\E$ means all the interaction terms in the
    inclusion-exclusion formula are estimated to be 0 but HRV improves
  on this.

Estimating $P[ Z^i > t^i ], \hskip 0.1 cm 1 \le i \le d,$ is a 
  standard procedure, perhaps using peaks over threshold  and
    maximum likelihood; see \cite[ page
  141]{dehaan:ferreira:2006}, \cite{coles:2001}. 
For $2 \le j \le d$,   $1 \le i_1 < i_2 < \cdots i_j \le d$, 
 large $k$ and large $n/k$, the
  probability $P[Z^{i_1} > t^{i_1}, \cdots, Z^{i_j} > t^{i_j} ]$ 
is  approximated using HRV on $\E^{(j)}$ by 
\begin{align}\label{risk_set_compute_firstapprox}
P[Z^{i_1} > t^{i_1}, \cdots, Z^{i_j} > t^{i_j} ] &= P\left[\frac{Z^{i_1}}{b^{(j)}(n/k)} > \frac{t^{i_1}}{b^{(j)}(n/k)}, \cdots, \frac{Z^{i_j}}{b^{(j)}(n/k)} > \frac{t^{i_j}}{b^{(j)}(n/k)} \right] \nonumber \\
& \approx \frac{k}{n}\nu^{(j)} \left( \left\{ {\bf{x}} \in \E^{(j)}: x^{i_1} > \frac{t^{i_1}}{b^{(j)}(n/k)}, \cdots, x^{i_j} > \frac{t^{i_j}}{b^{(j)}(n/k)} \right\} \right).
\end{align}
We need to estimate  $\nu^{(j)}$ and $b^{(j)}(n/k)$.
 Notice that, for $w^1, \cdots, w^j > 0$,
\begin{align}\label{risk_set_compute_semi}
\nu^{(j)} &\left( \left\{ {\bf{x}} \in \E^{(j)}: x^{i_1} > w^1,
    \cdots, x^{i_j} > w^j \right\} \right)
 \nonumber\\
&= \nu_{\alpha^{(j)}} \times S^{(j)}\left( \left\{ (r,{\boldsymbol{\theta}}) \in (0, \infty] \times  \delta \aleph^{(j)} : r{\theta}^{i_1} > w^1, \cdots, r{\theta}^{i_j} > w^j \right\}\right)\nonumber \\
&= \int_{\delta \aleph^{(j)}}{\left(\vee_{p=1}^j \frac{w^{i_p}}{{\theta}^{i_p}} \right)}^{- \alpha^{(j)}}S^{(j)}(d{\boldsymbol{\theta}}).
\end{align}
Using \eqref{estimatorb2n/k}, we get
\begin{equation}\label{estimatorbjn/k}
Z^{(j)}_{(k)}{\big{/}} b^{(j)}\left(n/k\right) \stackrel{P}{\rightarrow} 1,
\end{equation}
and thus we use $Z^{(j)}_{(k)}$ as an estimator of
$b^{(j)}\left(n/k\right)$. From
 \eqref{risk_set_compute_firstapprox}, \eqref{risk_set_compute_semi}
 and \eqref{estimatorbjn/k}, we approximate  $P[Z^{i_1} > t^{i_1},
 \cdots, Z^{i_j} > t^{i_j} ]$ as 
\begin{equation*}
P[Z^{i_1} > t^{i_1}, \cdots, Z^{i_j} > t^{i_j} ] \approx \frac{k}{n} \int_{\delta \aleph^{(j)}}{\left(\bigvee_{p=1}^j \frac{t^{i_p}}{Z^{(j)}_{(k)}{\theta}^{i_p}} \right)}^{- \hat \alpha^{(j)}} \hat S^{(j)}(d{\boldsymbol{\theta}}),
\end{equation*}
where $\hat \alpha^{(j)}$ and $\hat S^{(j)}$ are the consistent estimates of $\alpha^{(j)}$ and $S^{(j)}$ obtained in Section \ref{E0l_stan_est_proc_sec}. 

\subsection{Linear combination of risks.}
A second kind of risk set used in hydrology
 \citep{bruun:tawn:1998, dehaan:deronde:1998} is
 of the form $\{ {\bf{x}} \in \E: \gamma_1x^1 + \gamma_2x^2 + \cdots +
 \gamma_dx^d > y \}$ for  $\gamma_i > 0, \hskip 0.1 cm i =1, 2,
 \cdots, d$ and $y > 0$. Here the risks could be wind speed and wave
 height and a linear combination represents dike exceedance. Assume for simplicity
 $d=2$ and  note
\begin{align}\label{lin_comb_decomp_eqn}
P[ \gamma_1Z^1 + \gamma_2Z^2 > y ] &= P[ \gamma_1Z^1 > y ] + P[ \gamma_2Z^2 > y ] -  P[ \gamma_1Z^1 > y,  \gamma_2Z^2 > y ] \nonumber \\
&\qquad + P[ \gamma_1Z^1 + \gamma_2Z^2 > y,  \gamma_1Z^1 \le y,  \gamma_2Z^2 \le y ] .
\end{align}
Suppose, ${\bf{Z}}, {\bf{Z}}_1, {\bf{Z}}_2, \cdots, {\bf{Z}}_n$ are
iid  vectors whose common distribution
 has non-standard regular variation on $\E =
\E^{(1)}$ as in \eqref{non-stan_reg_var_onE_stan} and HRV
 on $\E^{(2)}$ with scaling function $b^{(2)}(t)$ as
in \eqref{non-stan_hidden_reg_var_onE02}. 
Asymptotic independence holds and thus 
  regular variation on $\E$ estimates the last two terms
on the  right hand side of \eqref{lin_comb_decomp_eqn} 
as zero. This
is crude and HRV should improve the risk estimate.

As in the previous scenario, 
estimating $P[ \gamma_iZ^i > y ], \, i= 1, 2,$ using
\eqref{mar_in_non-stan_case} is standard and we
proceed to estimate $P[
\gamma_1Z^1 > y,  \gamma_2Z^2 > y ].$ From Section 2.3 in
\cite{heffernan:resnick:2005}, we have
\eqref{non-stan_hidden_reg_var_onE02} equivalent to  
\begin{equation}\label{equi_non-stan_condn}
tP\left[ \left( \frac{ Z^j}{a^j(b^{(2)}(t))}, j =1,2 \right) \in \cdot \,
\right] \stackrel{v}{\rightarrow} \tilde \nu^{(2)} (\cdot) \quad \text{in $M_+(\E^{(2)})$},
\end{equation}
where $\tilde \nu^{(2)}$ and $\nu^{(2)}$ are related by 
\begin{equation}\label{equi_nonstan_limit_meas_condn}
\tilde \nu^{(2)}(({\bf{x}}, {\boldsymbol{\infty}}]) = \nu^{(2)}(({\bf{x}^{\boldsymbol{\beta}}}, {\boldsymbol{\infty}}]), \hskip 0.2 cm {\bf{x}} \in \E^{(2)},
\end{equation} 
where ${\boldsymbol{\beta}} = (\beta^1, \beta^2)$ and $\beta^j, \hskip
0.1 cm j = 1,2,$ is the marginal index of regular variation
defined in  
\eqref{mar_in_non-stan_case}. Using \eqref{equi_non-stan_condn} and 
\eqref{equi_nonstan_limit_meas_condn}, we approximate $P[ \gamma_1Z^1 > y,
\gamma_2Z^2 > y ]$  as 
\begin{align}\label{non-stan_risk_set_compute_firstapprox}
P[\gamma_1Z^{1} > y, \gamma_2Z^{2} &> y ] = P\left[\frac{Z^1}{a^1(b^{(2)}(n/k))} > \frac{y}{\gamma_1a^1(b^{(2)}(n/k))}, \frac{Z^2}{a^2(b^{(2)}(n/k))} > \frac{y}{\gamma_2a^2(b^{(j)}(n/k))} \right] \nonumber \\
& \approx \frac{k}{n} \tilde \nu^{(2)} \left( \left\{ {\bf{x}} \in \E^{(2)}: x^1 > \frac{y}{\gamma_1a^1(b^{(2)}(n/k))}, x^2 > \frac{y}{\gamma_2a^2(b^{(j)}(n/k))} \right\} \right) \nonumber\\
& = \frac{k}{n} \nu^{(2)} \left( \left\{ {\bf{x}} \in \E^{(2)}: x^1> {\left(\frac{y}{\gamma_1a^1(b^{(j)}(n/k))}\right)}^{\beta^1} , x^2 > {\left(\frac{y}{\gamma_2a^2(b^{(j)}(n/k))}\right)}^{\beta^2} \right\} \right).
\end{align}

We require estimates of $\nu^{(2)}, \beta^i$
and $a^i(b^{(2)}(n/k)), \hskip 0.1 cm i =1,2.$ There are standard
methods for estimating one dimensional indices $\beta^i,
 i = 1, 2,$ based on \eqref{mar_in_non-stan_case} 
(\cite[Chapter 4]{resnickbook:2007}, \cite{coles:2001, dehaan:ferreira:2006})
which yield consistent estimators $\hat \beta^i,
i = 1, 2$. For
$\nu^{(2)},$ observe,
\begin{align}\label{non-stan_risk_set_compute_semi} 
\nu^{(2)} \Bigl( \Bigl\{ {\bf{x}} \in \E^{(2)}& : x^1 > w^1, x^2 > w^2
  \Bigr\} \Bigr) 
\nonumber\\ 
&= \nu_{\alpha^{(2)}} \times S^{(2)}\left( \left\{
    (r,{\boldsymbol{\theta}}) \in (0, \infty] \times  \delta
    \aleph^{(2)} : r{\theta}^1 > w^1, r{\theta}^2 > w^2
  \right\}\right)
\nonumber \\ 
&= \int_{\delta \aleph^{(2)}}{\left(\frac{w^1}{{\theta}^1}\bigvee \frac{w^2}{{\theta}^2}
  \right)}^{- \alpha^{(2)}}S^{(2)}(d{\boldsymbol{\theta}}), \qquad
\qquad (w^1, w^2 > 0). 
\end{align}

Also, from Section 4.3 in \cite{heffernan:resnick:2005}, we get
\begin{equation}\label{non-stan_estimatorbjn/k}
Z^j_{(\lceil 1/m^{(2)}_{(k)} \rceil ) } {\big{/}} a^j\left(b^{(2)}\left(n/k \right)\right) \stackrel{P}{\rightarrow} 1,
\end{equation}
where $Z^j_{(\lceil 1/m^{(2)}_{(k)} \rceil ) }$ is the $\lceil
1/m^{(2)}_{(k)} \rceil $-th largest order statistic of the $j$-th
components of ${\bf{Z}}_i, \hskip 0.1 cm i=1, 2, \cdots, n.$ So, we
 use $Z^j_{(\lceil 1/m^{(2)}_{(k)} \rceil ) }$ as an estimator of
 $a^j\left(b^{(2)}\left(n/k\right)\right), \hskip 0.2 cm j =1,
 2$. Finally, using \eqref{non-stan_risk_set_compute_firstapprox},
 \eqref{non-stan_risk_set_compute_semi} and
 \eqref{non-stan_estimatorbjn/k}, we approximate  $P[\gamma_1Z^1 > y,
 \gamma_2Z^2 > y ]$ as 
\begin{equation}\label{finalrisksetestimator}
P[\gamma_1Z^1 > y, \gamma_2Z^2 > y ] \approx \frac{k}{n} \int_{\delta \aleph^{(2)}}{\left(\vee_{p=1}^2 \frac{1}{\theta^p} \cdot {\left( \frac{y}{\gamma_p Z^p_{(\lceil 1/m^{(2)}_{(k)} \rceil ) }} \right)}^{\hat \beta^p} \right)}^{- \hat \alpha^{(2)}} \hat S^{(2)}(d{\bf{\theta}}),
\end{equation}
where $\hat \alpha^{(2)}$ and $\hat S^{(2)}$ are consistent estimates of $\alpha^{(2)}$ and $S^{(2)}$  obtained in Section \ref{detect_nonstan_e02}.

Estimation of the fourth term of the right side of
\eqref{lin_comb_decomp_eqn}
requires care. First, observe
\begin{align}\label{the_trouble_set_eqn1}
P[ &\gamma_1Z^1 + \gamma_2Z^2 > y,  \gamma_1Z^1 \le y,  \gamma_2Z^2 \le y ] \nonumber \\
&= P\left[ \gamma_1a^1(b^{(2)}(n/k)) \frac{Z^1}{a^1(b^{(2)}(n/k))} + \gamma_2a^2(b^{(2)}(n/k)) \frac{Z^2}{a^2(b^{(2)}(n/k))} > y, \right. \nonumber\\
&\hskip 2 cm    \left. \gamma_1a^1(b^{(2)}(n/k)) \frac{Z^1}{a^1(b^{(2)}(n/k))} \le y, \gamma_2a^2(b^{(2)}(n/k)) \frac{Z^2}{a^2(b^{(2)}(n/k))} \le y \right] \nonumber\\
& \approx \frac{k}{n} \tilde \nu^{(2)}(\{ {\bf{x}} \in \E :
\gamma_1a^1(b^{(2)}(n/k)) x^1 + \gamma_2a^2(b^{(2)}(n/k)) x^2 > y, 
\nonumber\\
&\qquad \qquad \gamma_1a^1(b^{(2)}(n/k)) x^1 \le y, \gamma_2a^2(b^{(2)}(n/k)) x^2 \le y\} )\nonumber\\
& \approx \frac{k}{n} \tilde \nu^{(2)}(\{ {\bf{x}} \in \E :
\gamma_1Z^1_{(\lceil 1/m^{(2)}_{(k)} \rceil ) } x^1 +
\gamma_2Z^2_{(\lceil 1/m^{(2)}_{(k)} \rceil ) } x^2 > y, 
\gamma_1Z^1_{(\lceil 1/m^{(2)}_{(k)} \rceil ) } x^1 \vee \gamma_2Z^2_{(\lceil 1/m^{(2)}_{(k)} \rceil ) } x^2 \le y\} ).
\end{align}
In the last approximation in \eqref{the_trouble_set_eqn1}, $a^j\left(b^{(2)}\left(n/k \right)\right)$ is replaced by $Z^j_{(\lceil 1/m^{(2)}_{(k)} \rceil )}$, $j=1, 2,$ using \eqref{non-stan_estimatorbjn/k}. 

For  $\phi_1, \phi_2 > 0,$ the set $\{ {\bf{x}} \in \E:
\phi_1x^1 + \phi_2x^2 > y, \phi_1x^1 \le y, \phi_2x^2 \le y \} \subset
\E^{(2)}$ is not a compact subset of $\E^{(2)}$, so, $\tilde
\nu^{(2)}(\{ {\bf{x}} \in \E: \phi_1x^1 + \phi_2x^2 > y, \phi_1x^1 \le
y, \phi_2x^2 \le y \})$ could be infinite in which case
it is not clear how  HRV can
 refine the estimate of $P[
\gamma_1Z^1 + \gamma_2Z^2 > y,  \gamma_1Z^1 \le y,  \gamma_2Z^2 \le y
]$. So, we must check  finiteness of the  quantity on the
right side of \eqref{the_trouble_set_eqn1}.  

Set
 $\phi_j = \gamma_jZ^j_{(\lceil 1/m^{(2)}_{(k)} \rceil ) },j =1,2$ and
define  $A:=\{ {\bf{x}} \in \E:
 \phi_1x^1 + \phi_2x^2 > y, \phi_1x^1 \le y, \phi_2x^2 \le y \}$. Using \eqref{equi_nonstan_limit_meas_condn}  and following
 similar methods as in \eqref{risk_set_compute_semi}, we get 
\begin{align}\label{chk_risk_set_finite_eqn}
\tilde \nu^{(2)}(\{ {\bf{x}} \in \E: &\phi_1x^1 + \phi_2x^2 > y, \phi_1x^1 \le y, \phi_2x^2 \le y \}) = \int_A \beta^1 \beta^2 (x^1)^{(\beta^1-1)}(x^2)^{(\beta^2-1)}\nu^{(2)}(d{\bf{x}}) \nonumber \\
&= \int_{\delta \aleph^{(2)}} \beta^1 \beta^2 (\theta^1)^{(\beta^1-1)}(\theta^2)^{(\beta^2-1)}\int_{y/(\phi_1\theta^1 + \phi_2\theta^2)}^{y /(\phi_1\theta^1 \vee \phi_2\theta^2)} r^{(\beta^1 +\beta^2 -2)} \nu_{\alpha^{(2)}}(dr) S^{(2)}(d{\boldsymbol{\theta}}) \nonumber \\
&= \int_{\delta \aleph^{(2)}} \frac{\beta^1 \beta^2}{\beta^1 + \beta^2 -\alpha^{(2)} -2} (\theta^1)^{(\beta^1-1)}(\theta^2)^{(\beta^2-1)} \nonumber \\
& \qquad \times \left[ {\left(\frac{y}{\phi_1\theta^1 + \phi_2\theta^2}\right)}^{(\beta^1 +\beta^2 - \alpha^{(2)} -2)} - {\left(\frac{y}{\phi_1\theta^1 \vee \phi_2\theta^2}\right)}^{(\beta^1 +\beta^2 - \alpha^{(2)} -2)}\right] S^{(2)}(d{\boldsymbol{\theta}}).
\end{align}
Finiteness of the quantity on the right hand side of
\eqref{the_trouble_set_eqn1} 
is equivalent to the finiteness of the quantity on the right
hand side of \eqref{chk_risk_set_finite_eqn} 
which is
difficult to verify; see 
\cite{heffernan:resnick:2005}.
This problem is inherent in estimation for this type of
risk region.

We proceed assuming the finiteness of $\tilde \nu^{(2)}(\{
{\bf{x}} \in E: \phi_1x^1 + \phi_2x^2 > y, \phi_1x^1 \le y, \phi_2x^2
\le y \})$. From \eqref{the_trouble_set_eqn1} and
\eqref{chk_risk_set_finite_eqn}, we get for large $k$ and $n/k$,
the  estimate,
\begin{align*}
P[ \gamma_1Z^1 + & \gamma_2Z^2 > y,  \gamma_1Z^1 \le y,  \gamma_2Z^2 \le y ] \approx \frac{k}{n} \int_{\delta \aleph^{(2)}} \frac{ \hat \beta^1 \hat \beta^2}{\hat \beta^1 + \hat \beta^2 - \hat \alpha^{(2)} -2} (\theta^1)^{(\hat \beta^1-1)}( \theta^2)^{( \hat \beta^2-1)} \\
&\times \left[ {\left(\frac{y}{\phi_1\theta^1 + \phi_2\theta^2}\right)}^{(\hat \beta^1 + \hat \beta^2 - \hat \alpha^{(2)} -2)} - {\left(\frac{y}{ \phi_1\theta^1 \vee \phi_2\theta^2}\right)}^{(\hat \beta^1 + \hat \beta^2 - \hat \alpha^{(2)} -2)}\right] \hat S^{(2)}(d{\boldsymbol{\theta}}),
\end{align*}
where $\hat \alpha^{(2)}$ and $\hat S^{(2)}$ are consistent estimates of $\alpha^{(2)}$ and $S^{(2)}$ obtained in Section \ref{detect_nonstan_e02}.

\section{Computational Examples}\label{sec:eg}
This section 
considers the performance of the estimation
procedure described in
Section \ref{risksetcomputation} on 
two data sets, one simulated and one consisting of Internet
  measurements.
We also compare performance with
Heffernan and Resnick \cite{heffernan:resnick:2005}.  

\subsection{Simulated data}\label{subsec:sim}
We simulated iid samples  $\{ (X_i, Y_i), i = 1, 2, \cdots, n =
5000.\}$, where $X_1 \sim \rm{Pareto}(1),$ $ Y_1 \sim \rm{Pareto}(2)$
and $X_1$ and $Y_1$ are independent. Therefore, using
\eqref{non-stan_hidden_reg_var_onE02} we get  
\begin{equation}\label{simulateddatanu2}
\nu^{(2)}\left( (x,y), \boldsymbol{\infty} \right] =
\frac{1}{xy},\qquad \qquad (x, y > 0),
\end{equation}
and $\alpha^{(2)} = 2$ and $\nu^{(2)}( \{ {\bf{x}} \in \E^{(2)}
: x^{(2)} \ge 1, x^{(1)} = \infty \} ) = 0.$ Using
\eqref{s_and_tildes}, we obtain  the transformed hidden spectral
measure $\tilde S^{(2)}$  
\begin{equation}\label{tildes_computation}
\tilde S^{(2)} (\cdot) = \nu^{(2)}\left(\left\{ {\bf{x}} \in \E^{(2)} : x^{(2)} \ge 1, \frac{x^2}{x^1 + x^2} \in \cdot \right\} \right).
\end{equation}
The density with respect to Lebesgue measure of $\tilde S^{(2)}$ is
\begin{equation}\label{actual_tr_sp_density}
f (s) = \left\{ \begin{array}{cc} \frac{1}{2} {(1-s)}^{-2}, & \hbox{ if $0 \le s < \frac{1}{2},$}\\
\frac{1}{2} s^{-2}, & \hbox{ if $\frac{1}{2} \le s \le 1,$ }\\
0 & \hbox{otherwise.} 
\end{array} \right.
\end{equation}

\begin{figure}
\centering
\scalebox{0.8}
{\includegraphics{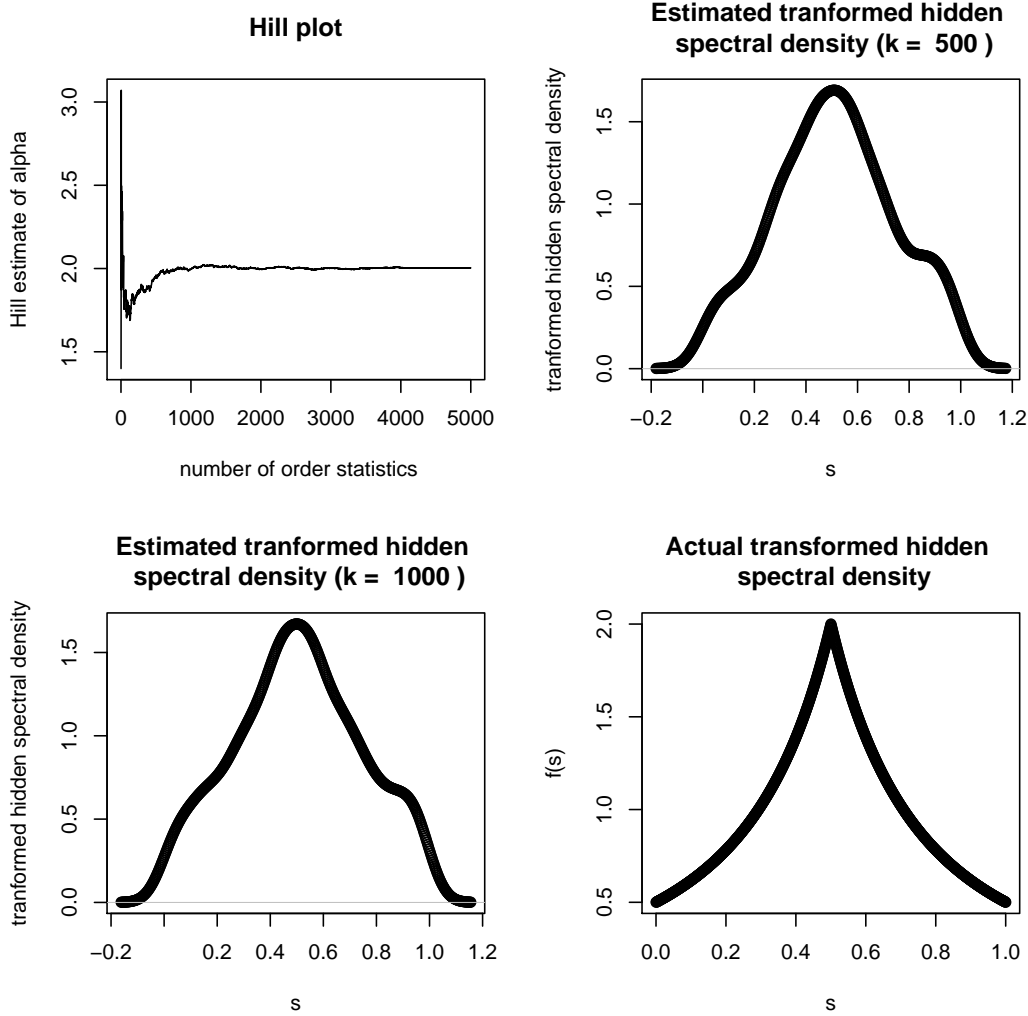}}
\caption{Hill plot of $\alpha^{(2)}$ and estimated and actual transformed hidden spectral densities }
\label{simulation_hill_spec}
\end{figure}

We test accuracy of
 our estimates of $\alpha^{(2)}$ and $\tilde S^{(2)}(\cdot)$.
 The Hill plot for $\{m^{(2)}_i, i = 1, 2, \cdots, n
\},$ the plot of the
estimated transformed hidden spectral densities for
$k = 500, 1000$, and the plot of the actual transformed hidden
spectral density  \eqref{actual_tr_sp_density} are shown in
Figure \ref{simulation_hill_spec}. 
We also estimate probabilities of risk sets of the form $P[ X_1
> t_1, Y_1 > t_2]$ for large thresholds $t_1$ and $t_2$. Using a
method similar to the one used to obtain
\eqref{finalrisksetestimator}, we estimate the probability $P[ X_1 >
t_1, Y_1 > t_2]$ as 
\begin{equation*}
P[X_1 > t_1, Y_1 > t_2 ] \approx \frac{k}{n} \int_{\delta
  \aleph^{(2)}}{\Bigl( \Bigl[ \frac{1}{\theta^1}  {\Bigl(
      \frac{t_1}{ X_{(\lceil 1/m^{(2)}_{(k)} \rceil ) }}
    \Bigr)}^{\hat \beta^1} \Bigr] \bigvee \Bigl[\frac{1}{\theta^2}
   {\Bigr( \frac{t_2}{ Y_{(\lceil 1/m^{(2)}_{(k)} \rceil ) }}
    \Bigr)}^{\hat \beta^2} \Bigr]\Bigr)}^{- \hat \alpha^{(2)}} \hat
S^{(2)}(d{\bf{\theta}}), 
\end{equation*}
where $X_{(1)} \ge X_{(2)} \ge \cdots \ge X_{(n)}$ and $Y_{(1)} \ge
Y_{(2)} \ge \cdots \ge Y_{(n)}$ are the order statistics for $\{X_i,
i=1, \cdots, n\}$ and $\{Y_i, i=1, \cdots, n\}$, and the remaining notation has the same
meaning as in \eqref{finalrisksetestimator}. Since 
we simulate the data we take $\hat
\beta_1 = 1$ and $\hat \beta_2 = 2$ and concentrate on estimating $\alpha^{(2)}$
using the Hill estimator and estimating $S^{(2)}$ by the formula given
in \eqref{non-stan_spec_est}. 
 We compute
the estimates of  $P[X_1 > 100, Y_1 > \sqrt{10}]$ for different values
of $k$ using these estimators and plot the graph in Figure \ref{risk_probability_plot}. The
range of $k$ is $k = 500$ to $k= 5000.$ 

A different estimate of $P[ X_1 > t_1, Y_1 > t_2]$ is given by Heffernan and Resnick
\cite{heffernan:resnick:2005}:
\begin{equation*}
P[X_1 > t_1, Y_1 > t_2 ] \approx \frac{k}{n} \hat \nu^{(2)} \Bigl( \Bigl( \Bigl( {\Bigl(  \frac{t_1}{ X_{(\lceil 1/m^{(2)}_{(k)} \rceil ) }}  \Bigr)}^{\hat \beta^1}, {\Bigl( \frac{t_2}{ Y_{(\lceil 1/m^{(2)}_{(k)} \rceil ) }}  \Bigr)}^{\hat \beta^2} \Bigr) , \boldsymbol{\infty} \Bigr] \Bigr),
\end{equation*}
where $\hat \nu^{(2)}$ is defined in \eqref{non_stan_rank_conv}. We
again use $\hat \beta_1 =1, \hskip 0.1 cm \hat \beta_2 = 2$ and
estimate $\alpha^{(2)}$ using the Hill estimator. Then, using
the above estimator, we compute the probability $P[X_1 > 100, Y_1 >
\sqrt{10}]$ for different values of $k$ and plot it as a graph in
Figure \ref{risk_probability_plot}. The values of $k$ are chosen
between $k = 500$ and $k = 5000.$ 
 
Using the true distribution of $(X_1, Y_1),$ we calculate
$P[X_1 > 100, Y_1 > \sqrt{10}]=0.001.$ In Figure \ref{risk_probability_plot}, we observe that the plot of the risk estimates
 obtained using the Heffernan-Resnick \cite{heffernan:resnick:2005} estimator 
 is more stable but  our current estimator of $P[X_1
> 100, Y_1 > \sqrt{10}]$ is more accurate for most $k$ in the range
$k = 500$ to $k = 5000$.
 
 \begin{figure}
\centering
\scalebox{0.6}
{\includegraphics{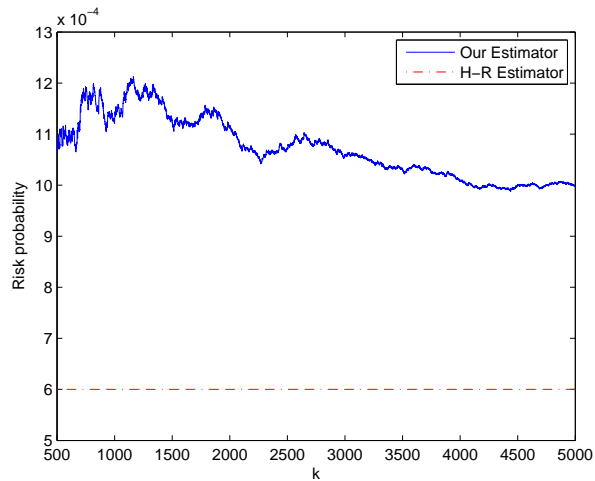}}
\caption{ Plots of estimates of $P[ X_1 > 100, Y_1 > \sqrt{10}]$ for different values of k (sample size = 5000) using both our  and H-R(Heffernan-Resnick \cite{heffernan:resnick:2005}) estimator}
\label{risk_probability_plot}
\end{figure}

The Heffernan-Resnick \cite{heffernan:resnick:2005} estimator of $P[X_1 > t_1, Y_1
  > t_2 ]$ uses an empirical distribution function 
and thus is subject to the defect that a zero estimate is
reported for the risk probability when $t_1$ and $t_2$ are  high but actually $P[X_1
> t_1, Y_1 > t_2 ]$ is non-zero. Irrespective of
how high the threshold is, our estimator does not estimate $P[X_1 >
t_1, Y_1 > t_2 ]$ as zero, unless it is
actually zero. 

As an illustration, we reduced the sample size to $n=500$
and applied the two estimators of the risk probability
$P[X_1 > 100, Y_1 > \sqrt{10}]$.
As suspected, the Heffernan-Resnick
\cite{heffernan:resnick:2005} estimator  estimates the
probability $P[X_1 > 100, Y_1 > \sqrt{10}]$ as zero, whereas our
estimator is still reasonably
accurate. This is shown in Figure
\ref{risk_probability_plot1} where $k$ ranges between
$k = 50$ and $k= 500.$ 

 \begin{figure}
\centering
\scalebox{0.6}
{\includegraphics{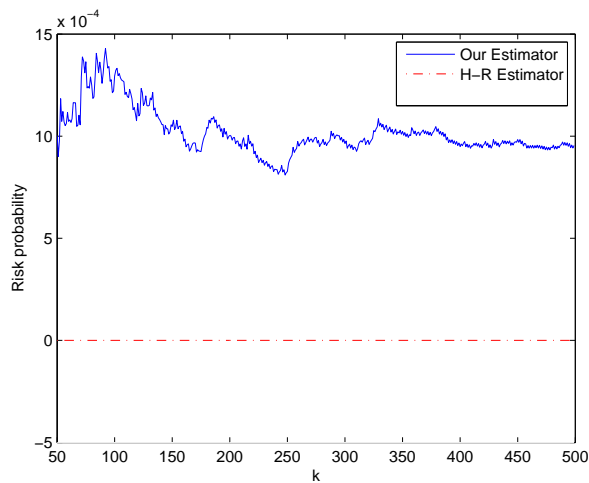}}
\caption{ Plots of estimates of $P[ X_1 > 100, Y_1 > \sqrt{10}]$ for different values of k (sample size = 500) using both our  and H-R(Heffernan-Resnick \cite{heffernan:resnick:2005}) estimator}
\label{risk_probability_plot1}
\end{figure}

\subsection{Internet traffic data} We analyze HTTP Internet response
data  consisting of sizes and durations of responses collected
during a four hour period from 1--5 pm on April 26, 2001 by the  
University of North Carolina at Chapel Hill
Department
of Computer Science's Distributed and Real-Time Systems Group under
the direction of Don Smith and Kevin Jaffey.  This dataset was also analyzed
in \cite{heffernan:resnick:2005}. 
We investigate 
joint behavior of two variables - size of response and throughput
(size of response/time duration of response)
and estimate the probability that both the size and rate are big
  as a  measure of burstiness. 

We start by estimating  marginal tail parameters. We use
QQ plots \citep[page 97]{resnickbook:2007}  (not shown here) to
choose the value $k =5000$ for both the variables size and rate.
Using this $k$,  we get the estimates of tail indices $\hat \beta^1 = 1.15$ and $\hat \beta^2 = 1.51$ for size and rate using the QQ estimator.

Next, we investigate presence of asymptotic independence
by plotting an estimated density of the
 transformed spectral measure $\tilde S( \cdot),$ defined in Section
 \ref{diff_par}.  In agreement with Heffernan and Resnick
 \cite{heffernan:resnick:2005}, 
our estimated density plots for different values of $k$
for  the transformed spectral measure  show two modes at the points $0$ and $1$ and
take values close to zero in between, thus indicating asymptotic
independence of size and rate (plots are not shown). 

Is hidden regular variation  present? The
Hill plot in Figure \ref{data_hill_spec}  of $\{ m^{(2)}_i,
1\leq i\leq n\}$ suggests this is so
and  we proceed to estimate the density of
the transformed hidden spectral measure. Figure \ref{data_hill_spec}
gives  plots of the estimated
transformed hidden spectral densities for $k = 500, 1000, 5000$.

\begin{figure}
\centering
\scalebox{0.8}
{\includegraphics{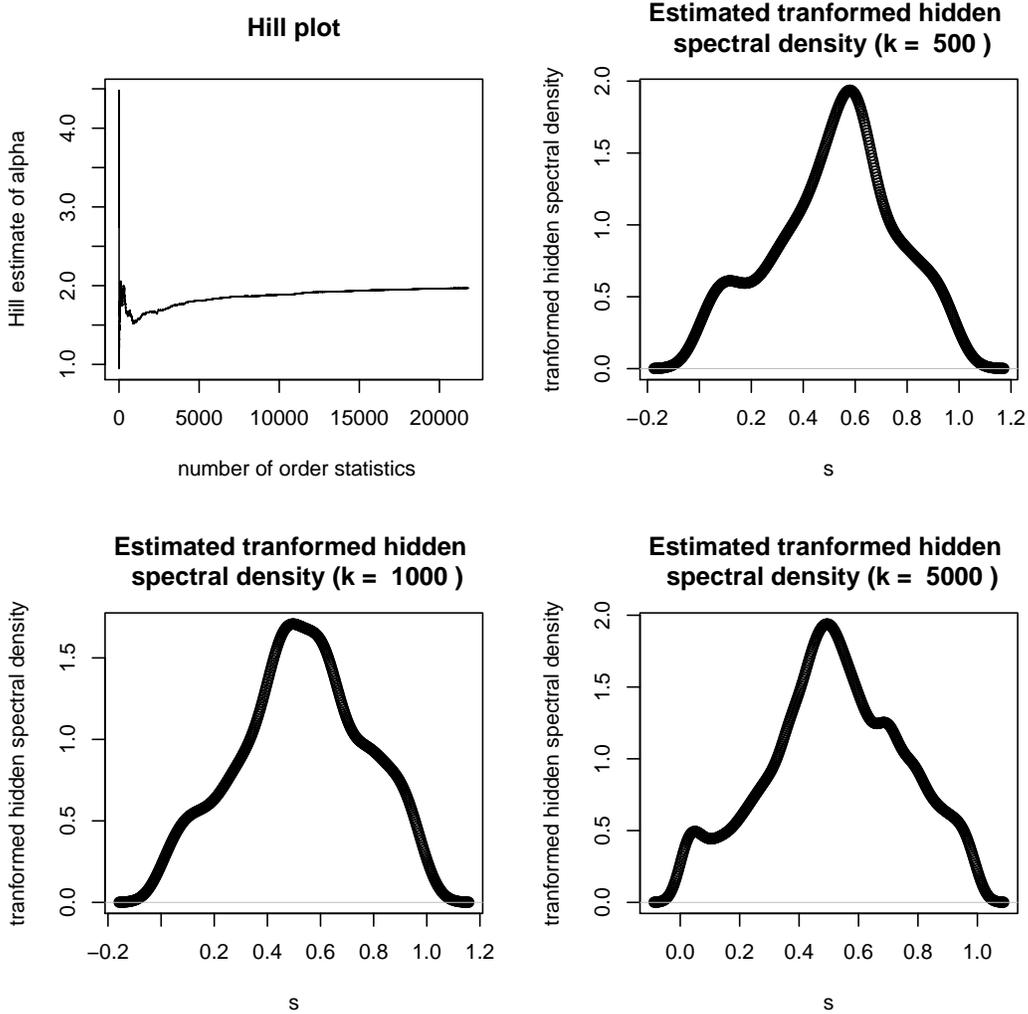}}
\caption{Hill plot of $\alpha^{(2)}$ and estimated transformed hidden spectral densities }
\label{data_hill_spec}
\end{figure}

Next, 
we  estimate  probabilities of risk sets of
the form $[ \text{Size} > x, \text{Rate} > y ]$ which we consider as measures
  of burstiness.
Examination of the (Size, Rate) data, 
indicates
$x = 2 \times 10^7$ and $y = 10^5$ are  reasonably high
thresholds. We use both our estimator and the estimator given in
\cite{heffernan:resnick:2005} to compute $P[ \text{Size} > x, \text{Rate} > y]$ for
different values of $k$ from $k=500$ to $k=n$ and plot them in Figure
\ref{risk_probability_plot2}.

 \begin{figure}
\centering
\scalebox{0.6}
{\includegraphics{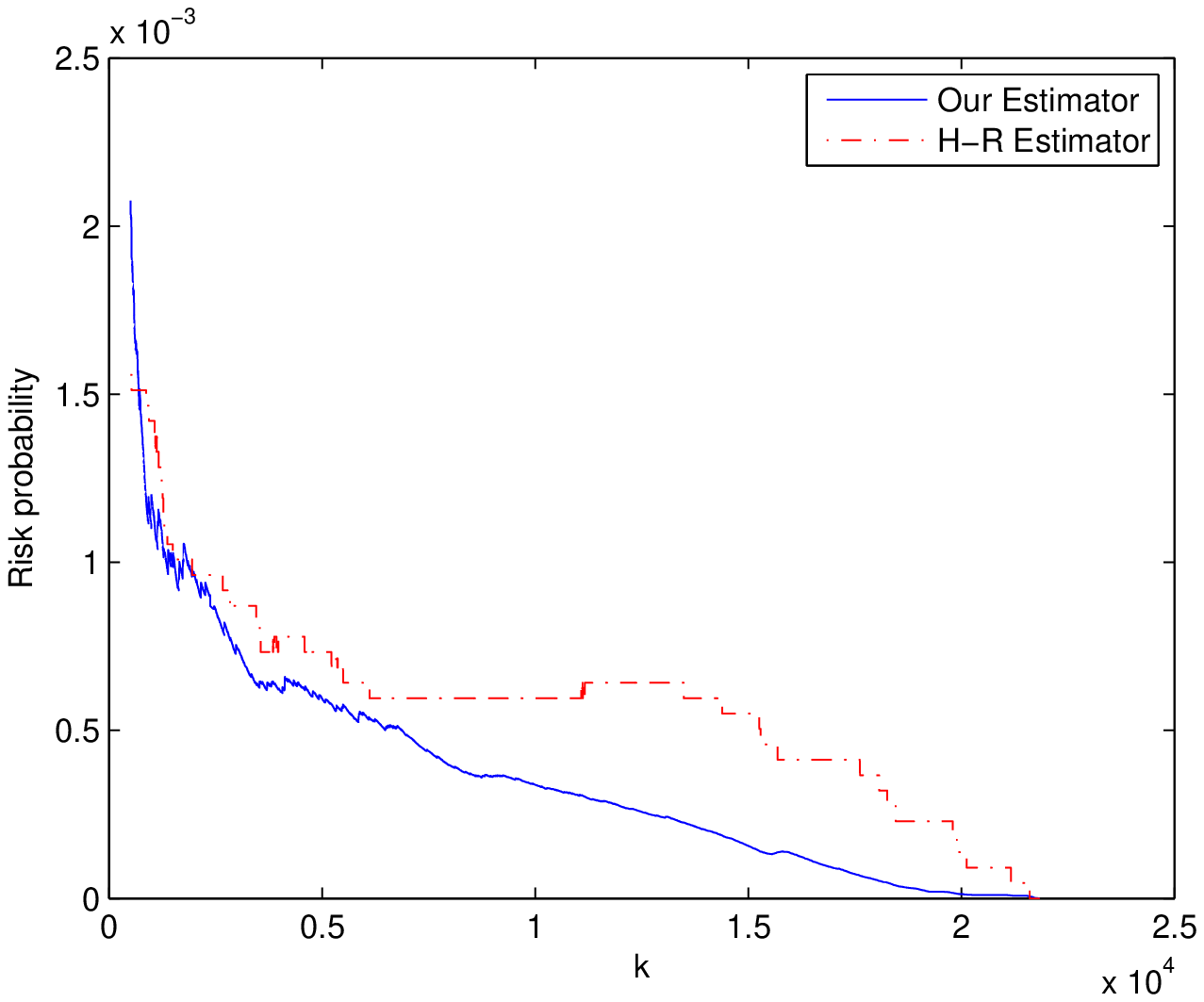}}
\caption{ Plots of estimates of $P[ \text{Size} > 2 \times 10^7, \text{Rate} > 10^5 ]$ for different values of $k$ using our estimator and H-R(Heffernan-Resnick \cite{heffernan:resnick:2005}) estimator}
\label{risk_probability_plot2}
\end{figure}

We also estimated $P[ \text{Size} > x, \text{Rate} > y ]$ for higher thresholds $x = 2
 \times 10^8$ and $y = 10^7,$ as a measure of extreme
traffic burstiness.  Again, we use both our estimator and the Heffernan-Resnick
 \cite{heffernan:resnick:2005}  estimator to estimate $P[ \text{Size}
 > x, \text{Rate} > y  ]$ 
for different values of $k$ from $k=500$ to $k=n$ and
 plot them in Figure \ref{risk_probability_plot3}. If hidden regular variation is present for the pair (Size, Rate), then the actual
risk probability cannot be zero. The Heffernan-Resnick
\cite{heffernan:resnick:2005} estimator reports an estimate of zero but
ours does not.

\begin{figure}
\centering
\scalebox{0.6}
{\includegraphics{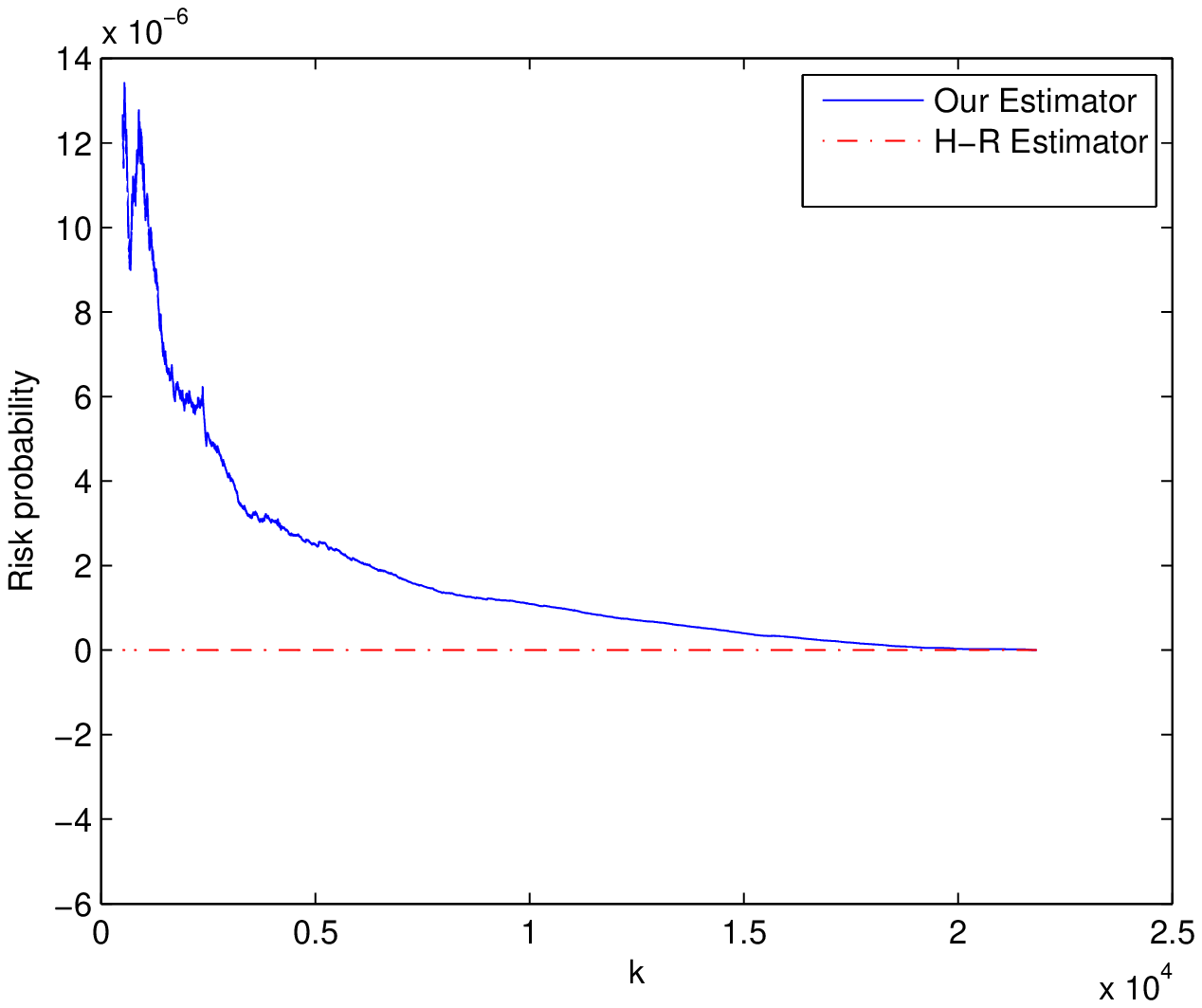}}
\caption{ Plots of estimates of $P[ \text{Size} > 2 \times 10^8, \text{Rate} > 10^6 ]$ for different values of $k$ using our estimator and H-R( Heffernan-Resnick \cite{heffernan:resnick:2005}) estimator}
\label{risk_probability_plot3}
\end{figure}

\section{Concluding remarks}\label{conclusion}
Hidden regular variation provides a sub-family of the distributions
having  regular variation on $\E$
that is sometimes equipped to 
obtain more precise estimates of probabilities
of certain risk sets, which are crudely estimated as zero by  regular
variation on $\E$; two examples are shown in Section
\ref{risksetcomputation}.  

The theory of HRV has deficiencies. Consider $d = 3$, and on
the planes of $\E^{(2)}$, suppose the random vector ${\bf{Z}}$ has regular
variation with three different tail indices $\alpha^{(2),1} <
\alpha^{(2),2} < \alpha^{(2),3}$. As a convention,
 say ${\bf{Z}}$ has regular variation
with tail index $\alpha^{(2), i}$ on $\{ {\bf{x}} \in \E^{(2)}: x^i =
0 \}$, $i =1 , 2, 3$. If we follow our HRV model and method of
estimation, we ignore the regular variation ${\bf{Z}}$ exhibits on $\{
{\bf{x}} \in \E^{(2)}: x^2 = 0 \}$ with tail index $\alpha^{(2),2}$,
which is actually more important than the regular variation on
$\E^{(3)}$. A way to repair this defect is the
following alternative method: In $d$
dimensions, first consider the big cone $\E$, then consider all
the $d \choose 2$ pairs of components of ${\bf{Z}}$ and their regular
variation on $(0, \infty]^2$, then consider all the $d \choose 3$
triplets of components of ${\bf{Z}}$ and their regular variation on
$(0, \infty]^3$, and so on. This alternative method requires considering
regular variation on $2^d -1$ cones, whereas our HRV formulation
 requires considering regular variation on at most $d$
cones. The alternative method is difficult to apply in high
  dimensions. Obviously, there is considerable flexibility in
choosing a nested sequence of cones and informed choice by a
practicioner will be governed by the application.

Another potential defect of our formulation of HRV  is that it is designed to deal
with the kind of degeneracy which arises when the limit measures are
concentrated on the axes, planes etc. But the limit measures might
exhibit different kind of degeneracies. For example, consider the
degeneracy in the case of complete asymptotic dependence, where the
limit measure is concentrated on the ray $\{ {\bf{x}} \in \E: x^1 =
x^2 = \cdots = x^d \}$. One might think of removing the ray and
considering hidden regular variation on the cone $\E \setminus \{
{\bf{x}} \in \E: x^1 = x^2 = \cdots = x^d \}$. Our HRV
  discussion does
not address this issue and  we are currently actively thinking about this
as well as methods of unifying theories of HRV and the conditional
extreme value model
\citep{
heffernan:tawn:2004,
heffernan:resnick:2007, das:resnick:2009,
das:resnick:2008b,
fougeres2010limit,
fougeres2009estimation}. 

Other variants of our formulation are possible. The hidden variation
on a subcone could be of extreme value type other than regular
variation and even if we focus only on the hidden variation being regular
variation, one could envisage different scaling functions for the
hidden variation. These topics are also being actively considered.

In estimating the limit measure $\nu^{(l)}$ of hidden regular
variation on $\E^{(l)}$, $2 \le l \le d$, we have suggested a method
that exploits the semi-parametric structure of $\nu^{(l)}$. Also, we
have constructed a consistent estimator of $\nu^{(l)}$ which relies
completely on non-parametric methods, as given in
\eqref{est_nu0_mostgeneral_e0l}. Our numerical experiments in
Section \ref{sec:eg}
clearly suggest  the method 
exploiting the semi-parametric structure is superior, presumably
because it uses more available information  about the limit 
measure $\nu^{(l)}$. However,  we have no precise, provable comparison.
                
An important statistical issue is we  have only
developed  parameter estimators which are consistent. We have not 
 yet developed  theory which allows one to report on confidence
 intervals for parameter estimates or risk probability estimates.
                    
For characterizations of hidden regular variation, it
is important to identify when  
$\nu^{(l)}(\{ {\bf{x}} \in \E^{(l)} : ||{\bf{x}}|| > 1 \})$ is
finite. We found a moment condition to check this, but it requires
knowledge of the hidden angular measure $S^{(l)}$. A similar problem
appeared in checking the finiteness of the right side of
\eqref{chk_risk_set_finite_eqn}. 
It would be useful to have a statistical test for finiteness.

\section{Acknowledgements}

S. I. Resnick and A. Mitra were partially supported by ARO
  Contract W911NF-07-1-0078 at Cornell University.

\bibliographystyle{plainnat.bst}
\bibliography{/Users/sidresnick/Documents/SidFiles/bibfile}

\begin{appendix}

\section{}\label{proof_stan_transform}

\begin{proof}[Proof of Proposition \ref{stan_transform}]

The idea of the proof is similar to Proposition 2 of \cite{resnick:2002a}. Define $\E_{l \setminus \infty} = \E^{(l)}\backslash \cup_{1 \le i_1 < i_2 < \cdots i_l \le d} [ x^{i_1} = \infty, x^{i_2} = \infty, \cdots, x^{i_l} = \infty]$ and $\E_2 = (0, \infty) \times \delta \aleph^{(l)}.$ Define a continuous bijection $Q^{(l)} : \E_{l\setminus \infty} \to \E_2$ as in \eqref{define_T_e0l}. We first show the equivalence of the vague convergence of measures restricted to $\E_{l\setminus \infty}$ and $\E_2$, and then extend the convergence to the corresponding whole spaces using the scaling property.\\
{\bf{Step 1:}}
First, we prove the direct part. So, we suppose that \eqref{reg_var_on_E0l} holds with $\nu^{(l)}(\aleph^{(l)}) = 1$. Hence, the convergence also holds with the measures being restricted to $\E_{l \setminus \infty}$, i.e.
\begin{equation*}
tP\left[ \frac{{\bf{Z}}}{b^{(l)}(t)} \in \cdot \cap \E_{l\setminus \infty} \right] \stackrel{v}{\rightarrow} \nu^{(l)}(\cdot \cap \E_{l\setminus \infty}) \quad \text{in $M_+(\E_{l\setminus \infty})$}.
\end{equation*}
Now, we proceed to show that for each compact set $K_2$ in $\E_2$, ${(Q^{(l)})}^{-1}(K_2)$ is a compact set of $\E_{l\setminus \infty}$. Note that the compact sets in $\E_{l\setminus \infty}$ are those closed sets $K$ for which every ${\bf{x}} \in K$ satisfies the property that $r \le x^{(l)} \le s$ for some $0 < r < s$ \cite[page 170]{resnickbook:2007}. Take a compact set $K_2$ in $\E_2$. We claim that $K_2$ must be contained in a set $\tilde K_2$ of the form 
$\tilde K_2 = [r, s] \times \delta \aleph^{(l)}$. Now, from the description of the compact sets of $\E_{l\setminus \infty}$, ${(Q^{(l)})}^{-1}(\tilde K_2) = \{ {\bf{x}} \in \E^{(l)} : r \le x^{(l)} \le s \}$ is compact in $\E_{l\setminus \infty}$. Also, since $Q^{(l)}$ is continuous, ${(Q^{(l)})}^{-1}(K_2)$ is closed.  Therefore, ${(Q^{(l)})}^{-1}(K_2)$ is a closed subset of the compact set ${(Q^{(l)})}^{-1}(\tilde K_2)$ and hence is compact in $\E_{l\setminus \infty}$.
So, using Proposition 5.5 (b) of \cite{resnickbook:2007} we get 
\begin{equation*}
tP\left[ \left( \frac{Z^{(l)}}{b^{(l)}(t)}, \frac{{\bf{Z}}}{Z^{(l)}} \right) \in \cdot \cap \E_2 \right] \stackrel{v}{\rightarrow} \nu_{\alpha^{(l)}} \times S^{(l)} ( \cdot \cap \E_2)  \quad \text{in $M_+(\E_2)$}.
\end{equation*}
Now, we want to extend the convergence over the whole space $(0, \infty] \times \delta \aleph^{(l)}.$ Choose any relatively compact subset $\Lambda$ of $\delta \aleph^{(l)}$ such that $S^{(l)}(\delta \Lambda) = 0$ and choose $s > r > 0$. Then,
\begin{align}\label{extn_1}
tP\left[ \frac{Z^{(l)}}{b^{(l)}(t)} > r, \frac{{\bf{Z}}}{Z^{(l)}} \in \Lambda \right] &\ge tP\left[ \frac{Z^{(l)}}{b^{(l)}(t)} \in (r, s], \frac{{\bf{Z}}}{Z^{(l)}} \in \Lambda \right] \to \nu_{\alpha^{(l)}}((r, s])S^{(l)}(\Lambda) \nonumber
\intertext{as $t \to \infty,$ which implies that for $s > r > 0$,}
\liminf_{t \to \infty} tP&\left[ \frac{Z^{(l)}}{b^{(l)}(t)} > r, \frac{{\bf{Z}}}{Z^{(l)}} \in \Lambda \right] \ge \nu_{\alpha^{(l)}}((r, s])S^{(l)}(\Lambda). \nonumber \\
\intertext{Hence, letting $s \to \infty$, we get}
\liminf_{t \to \infty} tP&\left[ \frac{Z^{(l)}}{b^{(l)}(t)} > r, \frac{{\bf{Z}}}{Z^{(l)}} \in \Lambda \right] \ge r^{-\alpha^{(l)}}S^{(l)}(\Lambda).
\end{align}
Now, we know that
\begin{align}\label{extn_2}
tP\left[ \frac{Z^{(l)}}{b^{(l)}(t)} > r, \frac{{\bf{Z}}}{Z^{(l)}} \in \Lambda \right] = tP\left[ \frac{Z^{(l)}}{b^{(l)}(t)} \in (r, s],  \frac{{\bf{Z}}}{Z^{(l)}} \in \Lambda \right] + tP\left[ \frac{Z^{(l)}}{b^{(l)}(t)} > s, \frac{{\bf{Z}}}{Z^{(l)}} \in \Lambda \right],
\end{align}
and
\begin{align}\label{extn_new1}
\lim_{ s \to \infty} \limsup_{t \to \infty}& tP\left[ \frac{Z^{(l)}}{b^{(l)}(t)} > s, \frac{{\bf{Z}}}{Z^{(l)}} \in \Lambda \right] \le \lim_{ s \to \infty} \limsup_{t \to \infty} tP\left[ \frac{Z^{(l)}}{b^{(l)}(t)} > s \right] \nonumber \\
&= \lim_{s \to \infty} \nu^{(l)}( \{ {\bf{x}} \in \E^{(l)} : x^{(l)} >s \}) 
= \lim_{s \to \infty} s^{-\alpha^{(l)}}\nu^{(l)}( \{ {\bf{x}} \in \E^{(l)} : x^{(l)} > 1 \}) = 0.
\end{align}
The first equality in the above set of relations follows from \eqref{reg_var_on_E0l}. Hence, from \eqref{extn_2} and \eqref{extn_new1}, we get 
\begin{equation}\label{extn_3}
\limsup_{ t \to \infty} tP\left[ \frac{Z^{(l)}}{b^{(l)}(t)} > r, \frac{{\bf{Z}}}{Z^{(l)}} \in \Lambda \right] \le \lim_{s \to \infty} \limsup_{t \to \infty}  tP\left[ \frac{Z^{(l)}}{b^{(l)}(t)} \in (r, s], \frac{{\bf{Z}}}{Z^{(l)}}  \in \Lambda \right] = r^{-\alpha^{(l)}}S^{(l)}(\Lambda).
\end{equation} 
Hence, from \eqref{extn_1} and \eqref{extn_3}, we conclude
the direct part of the proof:
\begin{equation*}
\lim_{ t \to \infty} tP\left[ \frac{Z^{(l)}}{b^{(l)}(t)} > r, \frac{{\bf{Z}}}{Z^{(l)}} \in \Lambda \right] = r^{-\alpha^{(l)}}S^{(l)}(\Lambda).
\end{equation*}

{\bf{Step 2:}} To see the converse, again we prove first the vague convergence of the restricted measures in $M_+(\E_{l \setminus \infty})$ and then extend it to convergence of measures in $M_+(\E^{(l)})$. We assume that \eqref{stan_case_transform} holds. Restriction on $\E_2$ gives 
\begin{equation*}
tP\left[ \left( \frac{Z^{(l)}}{b^{(l)}(t)}, \frac{{\bf{Z}}}{Z^{(l)}} \right) \in \cdot \cap \E_2\right] \stackrel{v}{\rightarrow} \nu_{\alpha^{(l)}} \times S^{(l)}(\cdot \cap \E_2) \quad \text{in $M_+(\E_2)$}.
\end{equation*}

 First we note that the compact sets of $\E_2$ are those closed sets $K$ for which every $(w, {\bf{v}}) \in K$ satisfies the property that $ r \le w \le s$ for some $0 < r < s.$ Take a compact set $K_1$ of $\E_{l \setminus \infty}$. Observe, from the description of the compact sets of $\E_{l \setminus \infty}$ as given before, that $K_1$ must be contained in a set $\tilde K_1$ of the form $\tilde K_1 = \{ {\bf{x}} \in \E^{(l)} : x^{(l)} \in [r, s] \}$. From the description of compact sets of $\E_2$, $Q^{(l)}(\tilde K_1) = [r, s] \times \delta \aleph^{(l)}$ is compact in $\E_2$. Since ${(Q^{(l)})}^{-1}$ is also continuous, the set $Q^{(l)}(K_1)$ is closed, and hence, being a closed subset of a compact set $Q^{(l)}(\tilde K_1)$, is compact. Therefore, using the continuous map ${(Q^{(l)})}^{-1}: \E_2 \to \E_{l \setminus \infty}$ and Proposition 5.5 (b) of \cite{resnickbook:2007}, we get
\begin{equation*}
tP\left[ \frac{{\bf{Z}}}{b^{(l)}(t)} \in \cdot \cap \E_{l \setminus \infty} \right] \stackrel{v}{\rightarrow} \nu^{(l)}(\cdot \cap \E_{l \setminus \infty}) \quad \text{in $M_+(\E_{l \setminus \infty})$}.
\end{equation*}
 
Now, we want to extend this convergence over the whole space $\E^{(l)}$. Choose a relatively compact set $A$ of $\E^{(l)}$ such that $\nu^{(l)}(\delta A) = 0$. Note that, from the description of relatively compact sets in $\E^{(l)}$ as given in Section \ref{formal_definition_model}, $A \subseteq \{ {\bf{x}} \in \E^{(l)} : x^{(l)} > r \}$ for some $r > 0$. Also, from the earlier description of compact sets of $\E_{l \setminus \infty}$, it follows that for all $s > r$, $A \cap \{ {\bf{x}} \in \E^{(l)} : x^{(l)} < s \}$ is a relatively compact set of $\E_{l \setminus \infty}$, and 
\begin{equation*}
\nu^{(l)}(\delta(A \cap \{ {\bf{x}} \in \E^{(l)} : x^{(l)} < s \} )) \le \nu^{(l)}(\delta A) + \nu^{(l)}(\{ {\bf{x}} \in \E^{(l)} : x^{(l)} = s \}) = 0.
\end{equation*}
Hence, it follows that 
\begin{align*}
tP\left[ \frac{{\bf{Z}}}{b^{(l)}(t)} \in A \right] \ge tP&\left[ \frac{{\bf{Z}}}{b^{(l)}(t)} \in A \cap \{ {\bf{x}} \in \E^{(l)} : x^{(l)} < s \} \right] \to \nu^{(l)} (A \cap \{ {\bf{x}} \in \E^{(l)} : x^{(l)} < s \}),
\end{align*}
which implies, by letting $s \to \infty$,
\begin{align}\label{extn2_1}
\liminf_{t \to \infty} tP&\left[ \frac{{\bf{Z}}}{b^{(l)}(t)} \in A \right] \ge \nu^{(l)} (A),
\end{align}
since by \eqref{scaling_e0l}, $\nu^{(l)}(\{ {\bf{x}} \in \E^{(l)} : x^{(l)} = \infty \}) = 0$.
Now, we know that
\begin{align}\label{extn2_2}
tP\left[ \frac{{\bf{Z}}}{b^{(l)}(t)} \in A \right] = tP&\left[ \frac{{\bf{Z}}}{b^{(l)}(t)} \in A \cap \{ {\bf{x}} \in \E^{(l)} : x^{(l)} < s \} \right] + tP\left[ \frac{{\bf{Z}}}{b^{(l)}(t)} \in A \cap \{ {\bf{x}} \in \E^{(l)} : x^{(l)} \ge s \} \right],
\end{align}
and
\begin{align}\label{extn2_new1}
\lim_{s \to \infty} \limsup_{t \to \infty} &tP\left[ \frac{{\bf{Z}}}{b^{(l)}(t)} \in A \cap \{ {\bf{x}} \in \E^{(l)} : x^{(l)} \ge s \} \right] \le \lim_{ s \to \infty} \limsup_{t \to \infty} tP\left[ \frac{{\bf{Z}}}{b^{(l)}(t)} \in \{ {\bf{x}} \in \E^{(l)} : x^{(l)} \ge s \} \right] \nonumber \\
&= \lim_{s \to \infty} \limsup_{t \to \infty} tP\left[ \frac{Z^{(l)}}{b^{(l)}(t)}  \ge s  \right] = \lim_{s \to \infty} \nu_{\alpha^{(l)}}([s, \infty]) = \lim_{s \to \infty} s^{-\alpha^{(l)}} = 0.
\end{align}
The second equality in the above set of relations follows from \eqref{stan_case_transform}. Hence, from \eqref{extn2_2} and \eqref{extn2_new1}, we get
\begin{equation}\label{extn2_3}
\limsup_{ t \to \infty}tP\left[ \frac{{\bf{Z}}}{b^{(l)}(t)} \in A \right] \le \lim_{s \to \infty} \limsup_{t \to \infty} tP\left[ \frac{{\bf{Z}}}{b^{(l)}(t)} \in A \cap \{ {\bf{x}} \in \E^{(l)} : x^{(l)} < l \} \right] = \nu^{(l)} (A).
\end{equation} 
Therefore, from \eqref{extn2_1} and \eqref{extn2_3}, we conclude
\begin{equation*}
\lim_{ t \to \infty} tP\left[ \frac{{\bf{Z}}}{b^{(l)}(t)} \in A \right] = \nu^{(l)} (A),
\end{equation*}
which completes the converse part of the proof.

\end{proof}

\section{}\label{proof_prop_spec_most_general}

\begin{proof}[Proof of Proposition \ref{spec_trans_mostgeneral}] 

The idea of this proof is similar to Theorem 6.1 of \cite[page 173]{resnickbook:2007}.
Define,
$\E_{l \setminus \infty} = \E^{(l)} \backslash \cup_{1 \le i_1 < i_2 < \cdots < i_l \le d} [ x^{i_1} = \infty, x^{i_2} = \infty, \cdots, x^{i_l} = \infty] $ and $\E_2 = (0, \infty)\times  \delta \aleph^{(l)}$. Now, define the continuous bijection $Q^{(l)} : \E_{l \setminus \infty} \to \E_2$ as in \eqref{define_T_e0l}. As in Proposition \ref{stan_transform}, here also the idea of the proof is to first show the equivalence of the weak convergence of random measures restricted to 
$M_+(\E_{l \setminus \infty})$ and $M_+(\E_2)$, and then extend the convergence to the corresponding whole spaces using the scaling property.

{\bf{Step 1:}} First, we prove that \eqref{non_stan_rank_conv} implies \eqref{non_stan_rank_trans_conv}. The convergence in \eqref{non_stan_rank_conv} implies 
\begin{equation*}
 \hat \nu^{(l)}( \cdot  \cap \E_{l \setminus \infty} ) := \frac{1}{k}\sum_{i =1}^n \epsilon_{\left( (1/r_i^j)/ m^{(l)}_{(k)}, 1 \le j \le d \right)}1_{\{m^{(l)}_i/ m^{(l)}_{(k)} < \infty\}} \Rightarrow \nu^{(l)}( \cdot  \cap \E_{l \setminus \infty}) 
 \end{equation*}
on $M_+(\E_{l \setminus \infty})$. Also, as shown in the proof of Proposition \ref{stan_transform}, for any compact set $K_2 \subset \E_2$, ${(Q^{(l)})}^{-1}(K_2)$ is compact in $\E_{l \setminus \infty}$. Then, using Proposition 5.5(b) of \cite{resnickbook:2007} we get
\begin{equation}\label{conv_restricted}
 \hat \nu^{(l)}(\cdot  \cap \E_{l \setminus \infty} ) \circ {(Q^{(l)})}^{-1} := \frac{1}{k} \sum_{i=1}^n \epsilon_{ \left(m^{(l)}_i/ m^{(l)}_{(k)}, \hskip 0.2 cm \left((1/r_i^j)/ m^{(l)}_i, \hskip 0.1 cm 1 \le j \le d \right) \right) }1_{\{m^{(l)}_i/ m^{(l)}_{(k)} < \infty\}}  \Rightarrow \nu_{\alpha^{(l)}} \times S^{(l)}( \cdot \cap \E_2 ) 
\end{equation}
on $M_+(\E_2)$.
To extend the convergence to the space $(0, \infty] \times \delta \aleph^{(l)}$, we use the convergence of laplace functionals and use Theorem 5.2 of \cite[page 137]{resnickbook:2007}. Take $f \in C_K^+( (0, \infty] \times \delta \aleph^ {(l)}) $, where $C_K^+(\F)$ is the set of all continuous functions with compact support  from $\F$ to $\R^+$. To relate this function to one defined in $C_K^+( (0, \infty) \times \delta \aleph^{(l)} )$, for all $\delta, M > 0$, we define a truncation function 
\begin{equation*}
\phi_{\delta, M} = \left \{ \begin{array}{cc} 1 & \hbox{if} \hskip 0.1 cm 0 < t \le M \\ 0 & \hbox{if} \hskip 0.1 cm t > M + \delta, \\
\hbox{linear interpolation} & \hbox{if} \hskip 0.1 cm M < t \le M + \delta. \end{array}\right.
\end{equation*}
Note that $f_{\delta, M}(r, \theta) := f(r, \theta) \phi_{\delta, M}(r) \in C_K^+( (0, \infty) \times \delta \aleph^{(l)} )$ for all $\delta, M > 0$. Note that
\begin{align*}
| E &\left[ \exp\left[ - \frac{1}{k}\sum_{i=1}^n f \left(m^{(l)}_i/ m^{(l)}_{(k)}, \hskip 0.1 cm \left((1/r_i^j)/ m^{(l)}_i, \hskip 0.1 cm 1 \le j \le d \right) \right)  \right] \right] - \exp \left[ -  \nu_{\alpha^{(l)}} \times S^{(l)}(f) \right] |\\
& \le |E\left[ \exp\left[ - \frac{1}{k}\sum_{i=1}^n f\left(m^{(l)}_i/ m^{(l)}_{(k)}, \hskip 0.2 cm \left((1/r_i^j)/ m^{(l)}_i, \hskip 0.1 cm 1 \le j \le d \right) \right)  \right] \right] \\
& \qquad - E \left[ \exp\left[ - \frac{1}{k}\sum_{i=1}^n f_{\delta, M} \left(m^{(l)}_i/ m^{(l)}_{(k)}, \hskip 0.2 cm \left((1/r_i^j)/ m^{(l)}_i, \hskip 0.1 cm 1 \le j \le d \right) \right)  \right] \right] | \\
+ & | E\left[ \exp\left[ - \frac{1}{k}\sum_{i=1}^n f_{\delta, M}\left(m^{(l)}_i/ m^{(l)}_{(k)}, \hskip 0.2 cm \left((1/r_i^j)/ m^{(l)}_i, \hskip 0.1 cm 1 \le j \le d \right) \right)  \right] \right] - \exp \left[ -  \nu_{\alpha^{(l)}} \times S^{(l)} (f_{\delta, M}) \right]|\\
& +  | \exp \left[ -  \nu_{\alpha^{(l)}} \times S^{(l)} (f_{\delta, M}) \right] - \exp \left[ -  \nu_{\alpha^{(l)}} \times S^{(l)} (f) \right]| \\
& \qquad = A + B + C.
\end{align*}
Since, $f_{\delta, M} \in C_K^+( (0, \infty) \times \delta \aleph^ {(l)} )$, by \eqref{conv_restricted}, we get $\lim_{n \to \infty} B = 0$. Now, we proceed to show that $\lim_{M \to \infty} \limsup_{n \to \infty} A = 0$. Notice that
\begin{align*}
\limsup_{n \to \infty} |E&\left[ \exp \left[ - \frac{1}{k}\sum_{i=1}^n f\left(m^{(l)}_i/ m^{(l)}_{(k)}, \hskip 0.2 cm \left((1/r_i^j)/ m^{(l)}_i, \hskip 0.1 cm 1 \le j \le d \right) \right)  \right] \right] \\
 &\qquad - E \left[ \exp\left[ - \frac{1}{k}\sum_{i=1}^n f_{\delta, M} \left(m^{(l)}_i/ m^{(l)}_{(k)}, \hskip 0.2 cm \left((1/r_i^j)/ m^{(l)}_i, \hskip 0.1 cm 1 \le j \le d \right) \right)  \right] \right] | \\
 &= \limsup_{n \to \infty}E\left[ \exp\left[ - \frac{1}{k}\sum_{i=1}^n f\left(m^{(l)}_i/ m^{(l)}_{(k)}, \hskip 0.2 cm \left((1/r_i^j)/ m^{(l)}_i, \hskip 0.1 cm 1 \le j \le d \right) \right)  \right] \right.\\
 &\hskip 1 cm  \times \left. \left( 1 - \exp \left[ -\frac{1}{k} \sum_{i=1}^n(f_{\delta, M} - f)\left(m^{(l)}_i/ m^{(l)}_{(k)}, \hskip 0.2 cm \left((1/r_i^j)/ m^{(l)}_i, \hskip 0.1 cm 1 \le j \le d \right) \right)\right] \right) \right] \\  
 & \le \limsup_{n \to \infty}E\left[  \left( 1 - \exp \left[ -\frac{1}{k} \sum_{i=1}^n(f_{\delta, M} - f)\left(m^{(l)}_i/ m^{(l)}_{(k)}, \hskip 0.2 cm \left((1/r_i^j)/ m^{(l)}_i, \hskip 0.1 cm 1 \le j \le d \right) \right)\right] \right) \right],  \\
 \intertext{which, using the facts that $||f|| = \sup_{( r, \theta) \in (0, \infty] \times \delta \aleph^{(l)}} f(r, \theta) < \infty$, $|| f_{\delta, M} - f || \le ||f|| \cdot ||\phi_{\delta, M} - 1|| \le || f ||$ and $(f_{\delta, M} - f)(x, {\bf{\theta}}) = 0$ for $x < M$, is bounded by} 
 & \limsup_{n \to \infty} E \left[ 1 - \exp\left[ - \frac{1}{k} || f || \sum_{i=1}^n \epsilon_{m^{(l)}_i/ m^{(l)}_{(k)}} ([M, \infty]) \right] \right] \\
 &= E \left[ 1 - \exp\left[ - || f ||\frac{1}{k}\sum_{i=1}^n \epsilon_{\left( (1/r_i^j)/ m^{(l)}_{(k)}, 1 \le j \le d \right)} (\{ {\bf{x}} \in \E^{(l)} : x^{(l)} \in [M, \infty]\}) \right] \right],
 \intertext{which,  by \eqref{non_stan_rank_conv}, converges as $n \to \infty$, to}
 & 1 - \exp\left[ - || f || \cdot \nu^{(l)}(\{ {\bf{x}} \in \E^{(l)} : x^{(l)} \in [M, \infty]\}) \right] \\
 & = 1 - \exp\left[ - || f || \cdot M^{-\alpha^{(l)}} \nu^{(l)}( \aleph^{(l)}) \right] \\
 & = 1 - \exp\left[ - || f || \cdot M^{-\alpha^{(l)}} \right] \rightarrow 0,
\end{align*}
as $M \to \infty$. The argument for $\lim_{M \to \infty} C = 0$ is
similar and is
omitted. Hence, by Theorem 5.2 of \cite[page 137]{resnickbook:2007},
we obtain \eqref{non_stan_rank_trans_conv}.

{\bf{Step 2:}} To see the other part, i.e. \eqref{non_stan_rank_trans_conv} implies \eqref{non_stan_rank_conv}, we use a similar method. The convergence in \eqref{non_stan_rank_trans_conv} implies 
\begin{equation*}
 \frac{1}{k}\sum_{i =1}^n \epsilon_{\left( \left(m^{(l)}_i/ m^{(l)}_{(k)}, \hskip 0.1 cm (1/r_i^j)/ m^{(l)}_i \right), \hskip 0.1cm 1 \le j \le d \right)}((\cdot) \cap \E_2) \Rightarrow \nu_{\alpha^{(l)}} \times S^{(l)} ((\cdot) \cap \E_2)
\end{equation*}
on $M_+(\E_2)$. It is easy to see that ${(Q^{(l)})}^{-1}$ is a continuous bijection. Also, as shown in the proof of Proposition \ref{stan_transform}, for any compact set $K_1 \subset \E_{l \setminus \infty}$, $Q^{(l)}(K_1)$ is compact in $\E_2$. Therefore, using Proposition 5.5(b) of \cite{resnickbook:2007} we get 
\begin{equation}\label{conv_restricted_E_1}
 \hat \nu^{(l)}( (\cdot)  \cap \E_{l \setminus \infty} ) := \frac{1}{k}\sum_{i =1}^n \epsilon_{\left( (1/r_i^j)/ m^{(l)}_{(k)}, 1 \le j \le d \right)}1_{\{m_i^{(l)}/m_{(k)}^{(l)} < \infty\}} \Rightarrow \nu^{(l)}( (\cdot)  \cap \E_{l \setminus \infty})
 \end{equation}
 on $M_+(\E_{l \setminus \infty})$. We use the same truncation function $\phi_{\delta, M}$ to relate functions on $C_K^+(\E^{(l)})$ to ones in $C_K^+(\E_{l \setminus \infty})$. Choose $f \in C_K^+( \E^{(l)})$. Note that $f_{\delta, M}({\bf{x}}) := f({\bf{x}}) \phi_{\delta, M}(x^{(l)}) \in C_K^+( \E_{l \setminus \infty})$ for all $\delta, M > 0$.
 
 \begin{align*}
| &E \left[ \exp\left[ - \frac{1}{k}\sum_{i=1}^n f\left((1/r_i^j)/ m^{(l)}_{(k)}, \hskip 0.1 cm 1 \le j \le d \right)  \right] \right] - \exp \left[ -  \nu^{(l)}(f) \right]|\\
&\le |E\left[ \exp\left[ - \frac{1}{k}\sum_{i=1}^n f\left((1/r_i^j)/ m^{(l)}_i, \hskip 0.1 cm 1 \le j \le d \right)  \right]  - \exp\left[ - \frac{1}{k}\sum_{i=1}^n f_{\delta, M}\left((1/r_i^j)/ m^{(l)}_i, \hskip 0.1 cm 1 \le j \le d \right)  \right] \right] | \\
 & + | E\left[ \exp\left[ - \frac{1}{k}\sum_{i=1}^n f_{\delta, M}\left((1/r_i^j)/ m^{(l)}_i, \hskip 0.1 cm 1 \le j \le d \right) \right] \right] - \exp \left[ -  \nu^{(l)} (f_{\delta, M}) \right]|\\
& \qquad +  | \exp \left[ -  \nu^{(l)}(f_{\delta, M}) \right] - \exp \left[ -  \nu^{(l)} (f) \right]| \\
& \qquad = A + B + C.
\end{align*}
Since, $f_{\delta, M} \in C_K^+( \E_{l \setminus \infty})$, by \eqref{conv_restricted_E_1}, we get $\lim_{n \to \infty} B = 0$. Now, we will show that $\lim_{M \to \infty} \limsup_{n \to \infty} A = 0$.
\begin{align*}
\limsup_{n \to \infty} |E&\left[ \exp \left[ - \frac{1}{k}\sum_{i=1}^n f \left((1/r_i^j)/ m^{(l)}_i, \hskip 0.1 cm 1 \le j \le d \right)   \right] - \exp\left[ - \frac{1}{k}\sum_{i=1}^n f_{\delta, M} \left((1/r_i^j)/ m^{(l)}_i, \hskip 0.1 cm 1 \le j \le d \right)  \right] \right] | \\
 &= \limsup_{n \to \infty}E\left[ \exp\left[ - \frac{1}{k}\sum_{i=1}^n f\left((1/r_i^j)/ m^{(l)}_i, \hskip 0.1 cm 1 \le j \le d \right)  \right] \right.\\
 &\hskip 1 cm  \times \left. \left( 1 - \exp \left[ -\frac{1}{k} \sum_{i=1}^n(f_{\delta, M} - f)\left((1/r_i^j)/ m^{(l)}_i, \hskip 0.1 cm 1 \le j \le d \right) \right] \right) \right] \\  
 & \le \limsup_{n \to \infty}E\left[  \left( 1 - \exp \left[ -\frac{1}{k} \sum_{i=1}^n(f_{\delta, M} - f)\left((1/r_i^j)/ m^{(l)}_i, \hskip 0.1 cm 1 \le j \le d \right) \right] \right) \right],  \\
 \intertext{which, using the facts $||f|| = \sup_{{\bf{x}} \in \E^{(l)}} f({\bf{x}}) < \infty$, $|| f_{\delta, M} - f || \le ||f|| \cdot ||\phi_{\delta, M} - 1|| \le || f ||$ and $(f_{\delta, M} -f)({\bf{x}}) = 0$ for $\{ {\bf{x}} \in \E^{(l)}: x^{(l)} < M \}$, is bounded by} 
 & E \left[ 1 - \exp\left[ - \frac{1}{k} || f || \sum_{i=1}^n \epsilon_{\left( (1/r_i^j)/ m^{(l)}_{(k)}, 1 \le j \le d \right)} (\{ {\bf{x}} \in \E^{(l)} : x^{(l)} \in [M, \infty]\}) \right] \right]\\
 &= \limsup_{n \to \infty} E \left[ 1 - \exp\left[ - || f || \frac{1}{k} \sum_{i=1}^n \epsilon_{m^{(l)}_i/ m^{(l)}_{(k)}} ([M, \infty]) \right] \right] ,
 \intertext{which, by  \eqref{non_stan_rank_trans_conv}, converges as $n \to \infty$, to}
 & 1 - \exp\left[ - || f || \cdot \nu_{\alpha^{(l)}}([M, \infty]) \right] \\
 & = 1 - \exp\left[ - || f || \cdot M^{-\alpha^{(l)}}  \right] \rightarrow 0,
\end{align*}
as $M \to \infty$. The argument for $\lim_{M \to \infty} C = 0$ is
similar and is
 omitted. Hence, we obtain \eqref{non_stan_rank_conv} and this completes the proof.
\end{proof}

\end{appendix}

  \end{document}